\documentclass[11pt]{article}
\usepackage{amsmath, amssymb, amscd, amsthm, amsfonts}
\usepackage{graphicx}
\usepackage{hyperref}
\usepackage{indentfirst}
\usepackage[dvipsnames]{xcolor}

\newcommand{\Prob}{\mathrm{Prob}}
\newcommand{\supp}{\mathrm{supp}}
\newcommand{\bord}{\mathrm{Bord }}
\newcommand{\isom}{\mathrm{Isom}}
\newcommand{\overbar}[1]{\mkern 1.5mu\overline{\mkern-1.5mu#1\mkern-1.5mu}\mkern 1.5mu}
\definecolor{uuuuuu}{rgb}{0.26666666666666666,0.26666666666666666,0.26666666666666666}
\allowdisplaybreaks

\theoremstyle{plain}
\newtheorem{theorem}{Theorem} 
\newtheorem{proposition}{Proposition}
\newtheorem{corollary}{Corollary}
\newtheorem{lemma}{Lemma}

\theoremstyle{definition}
\newtheorem{definition}{Definition} 

\newtheorem{example}{Example} 
\newtheorem*{remark}{Remark}

\title{Regularity of the drift for random walks in groups acting on Gromov hyperbolic spaces}
\author {Luís Miguel Sampaio}

\begin{document}

	\maketitle
	
	\begin{abstract}
	    In this work we prove the continuity and existence of large deviations for the drift of random walks on closed separable groups acting by isometries on Gromov hyperbolic spaces. Through the process we refine the multiplicative ergodic theorem of Karlsson and Gouëzel \cite{karlsson2000hilbert} for such spaces. The works goes beyond what is known in the literature by allowing spaces that are not necessarily proper, although they satisfy some good behaviour at the boundary. 
	\end{abstract}
	
	\section{Introduction}
	
	A topic that naturally appears in many areas of mathematics is that of a product of random operators. Realistically, we cannot hope to obtain much information about the product itself as it may end up going to some notion of "infinity", so the natural solution has been to introduce  descriptive information that encodes the asymptotic behaviour of the product. A paradigmatic example of such an encoding quantity is given by the Lyapunov exponents and Oseledets filtration which, through the Oseledets multiplicative ergodic theorem, fully govern the dynamics of random products of matrices.
	
	Following Kingman's ergodic theorem \cite{kingman1968ergodic}, the typical way to device the existence of such average limit quantities tracking the product has been to guarantee they are subadditive. This is the case with Lyapunov exponents in the linear case as well as the drift for semicontractions in metric spaces. 
	
	In a random walk the successive operators are picked identically and independently  at each  step. A natural question to consider is what would happen if we slightly perturb the way in which the choices are made, either by choosing similar operators or changing the probability that each of them is picked. Our first goal in this text  is to prove the continuity of the drift in Gromov hyperbolic spaces with respect  to a random walk on its group  of isometries.
	
	Despite its remarkable applications, Kingman's theorem does not tell us anything about the rate of convergence towards the limiting quantities. The natural way to speak of this rate of convergence when considering average limit quantities is to introduce large deviations. With  that in mind, after presenting the continuity of the drift we focus  on obtaining large deviations for its associated convergence.
	
	In \cite{karlsson2000hilbert}, Karlsson and Gouëzel prove, using the notion of horofunction, that random products of semicontractions acting on a metric space follow a specific direction as they tend towards infinity. Our route towards obtaining large deviations will begin with refining this theorem in the scope of hyperbolic spaces, by fully describing what happens with every horofunction in what we call the hyperbolic multiplicative ergodic theorem (HMET). We will then follow the spectral method of Nagaev \cite{nagaev1957some} to obtain large deviations for the quantities appearing in HMET in an analogue fashion to that of Bougerol in \cite{bougerol1988theoremes}, Duarte and Klein in \cite{duarte2017Klein, duarte2016lyapunov} and Björklund in \cite{bjorklund2010central}. 
	
	Fortunately, the tools developed in the study of large deviations are also what will help us obtain continuity by an argument similar to Baraviera and Duarte in \cite{baraviera2019approximating} where they re-obtain Le Page's theorem \cite{page1989regularite}. For proper spaces, these results have been recently obtained by Aoun and Sert in \cite{aoun2021random}. 
	
	We dedicate the remainder of the introduction to introducing the concepts, problems and results. We will first introduce the geometric aspects and then we shall do our considerations towards random walks.
	
	\subsection{Geometric setting}
	
	Let $X$ be a metric space, define the Gromov product in $X$ as
	\begin{equation*}
		\langle x \, , \, z \rangle_y := \frac{1}{2}\left( d(x,y) + d(z,y) - d(x,z) \right) \hspace{1cm} \forall x,y,z \in X.
	\end{equation*}
	We say that $X$ is a Gromov $\delta$-hyperbolic space, or simply hyperbolic space, if for every $x,y,z$ and $w$ in $X$,
	\begin{equation}
		\label{4pcondition}
		\langle x \, , \, z \rangle_w \geq \min\{ \langle x \, , \, y \rangle_w \, , \, \langle y \, , \, z \rangle_w \} - \delta.
	\end{equation}
	We call (\ref{4pcondition}) the 4-point condition of hyperbolicity or Gromov's inequality.
	
	A metric space $X$ is said to be geodesic if for every two points $x$ and $y$ in $X$, there exists an isometric embedding $\gamma:[0,d(x,y)]\to X$ connecting $x$ to $y$. Throughout the text $X$ will denote a separable, geodesic although not necessarily proper hyperbolic metric space. For geodesic spaces, Gromov hyperbolicity has more geometric flavour (see \cite{das2017geometry}): $X$ is $\delta$-hyperbolic if there exists $\delta>0$ such that for every triangle in $X$, any side is contained in a $3\delta$-neighbourhood of the other two, in other words, geometrically, triangles are thin.
	
	In this text we will be interested in studying the behaviour of sequences approaching infinity in $X$, the natural way to deal with this convergence problem is to consider boundaries. Fortunately, hyperbolic spaces carry a natural boundary, called the Gromov boundary, whilst every metric space admits a compact boundary, called the horofunction compactification. A nice aspect of Gromov spaces is the interplay between these two boundaries, so let us briefly introduce them, although we explore them later with more detail in §\ref{Boundaries}. 
	
	We say that a sequence $(x_n)$ in a hyperbolic space $X$ with basepoint $x_0$ is a Gromov sequence if $\langle x_n\, , \, x_m \rangle_{x_0}$ tends to infinity as $m$ and $n$ tend to infinity. Two Gromov sequences $(x_n)$ and $(y_n)$ are equivalent, $(x_n)\sim (y_n)$, if $\langle x_n\, , \, y_n \rangle_{x_0}$ tends to infinity as $n$ tends to infinity. Gromov's inequality implies that this is an equivalence relation. The Gromov boundary, denoted by $\partial X$, is the set of equivalence classes of Gromov sequences.  Finally, $\isom(X)$, the group of isometries of $X$ naturally acts on $\partial X$ by sending $\xi = [x_n]_\sim$ to $g\xi = [gx_n]_\sim$. 

	The Gromov product in $X$ may be extended to its Gromov boundary: given $\xi, \eta \in \partial X$ and $y,z\in X$, let
	\begin{align*}
		\langle \xi \, , \, \eta \rangle_z & := \inf \left\{\liminf_{n,m\to \infty}\, \langle x_n \, , \, y_m \rangle_z \, : \, (x_n)\in \xi, \, (y_m)\in \eta \right\},\\
		\langle x \, , \, \xi \rangle_z &= \langle \xi \, , \, x \rangle_z := \inf \left\{\liminf_{n\to \infty} \, \langle x_n \, , \, x \rangle_z \, : \, (x_n)\in \xi \right\}.
	\end{align*}
	
	Denote by $\bord X$ the set $X\cup \partial X$. Given $1 < b \leq 2^\frac{1}{\delta}$ consider the symmetric map $\rho_b: \bord X \times \bord X \to \mathbb{R}$ given by
	\begin{equation*}
		\rho_b(\xi\, , \, \eta) = b^{-\langle \xi \, , \, \eta \rangle_{x_0}}.
	\end{equation*} 
	Using the Gromov inequality, for every $\xi,\eta,\zeta \in \bord X$
	\begin{equation*}
		\rho_b(\xi\, , \, \eta) \leq 2 \max \{ \rho_b(\xi\, , \, \zeta), \rho_b(\zeta\, , \, \eta)\}.
	\end{equation*}
	By the classic Frink metrization theorem (see \cite{frink1937distance}), the map $\bar{D}_b: \bord X \times \bord X\to \mathbb{R}$ given by
	\begin{equation*}
		\bar{D}_b(\xi \, , \, \eta) = \inf \sum_{i=0}^{n-1} \rho_b(\xi_i\, , \, \xi_{i+1}) 
	\end{equation*}
	where the infimum is taken over finite sequences of points $\xi_i$ such that $\xi_0= \xi$ and $\xi_n = \eta$, satisfies the triangle inequality. Moreover the following visibility condition holds:
	\begin{equation*}
		\rho_b(\xi\, , \, \eta)/4 \leq \bar{D}_b(\xi \, , \, \eta) \leq \rho_b(\xi\, , \, \eta) \textrm{ for every } \xi, \eta \in \partial X.
	\end{equation*}
	Restricting $\bar{D}_b$ to $\partial X$ yields a metric, as a consequence, $\partial X$ is bounded. Recall however that it won't be compact if the space is not proper. With $\bar{D}_b$ we can construct one additional metric in $\bord X$ given by the following proposition.
	
	\begin{proposition}[in \cite{das2017geometry}]
		For every $\xi, \eta \in \bord X$ let
		\begin{equation*}
			D_b(\xi, \eta) := \min \left\{ \log(b)d(\xi \, , \, \eta) \, ; \,  \bar{D}_b(\xi \, , \, \eta) \right\},
		\end{equation*}
		using the convention $d(\xi, \eta)=\infty$ if either $\xi$ or $\eta$ belong to $\partial X$. Then $D_b$ is a complete metric in $Bord X$, inducing in $X$ the same topology as the metric $d$.
	\end{proposition}

	From the visibility condition $D_b \leq 1$, in particular, $\bord X$ is a bounded space when equipped with this metric. The main drawback of this construction however is that in general $\bord X$ does not have to be compact. To combat this problem consider now the injection
	\begin{align*}
		\rho: X & \to C(X) \\
		x & \mapsto h_x( \cdot ) = d(\cdot , x) - d(x,x_0),
	\end{align*}
	where $C(X)\subset \mathbb{R}^X$ is endowed with the topology of pointwise convergence. Then  $X^h := \overbar{\rho(X)}$ is compact. We call the elements of $X^h$ horofunctions and  $X^h$ the horofunction compactification of $X$. The action of $\isom(X)$ on $X$ extends to its horofunction boundary as follows: for every $h \in X^h$, $g\in G$ and $z\in X$,
	\begin{equation}
		\label{horoAction}
		g\cdot h(z)= h(g^{-1}z)-h(g^{-1}x_0).
	\end{equation}
	The horofunction compactification can be partitioned into its finite and infinite $G$-invariant parts given, respectively, by $X_F^h := \{h\in X^h \, : \, \inf(h)>-\infty \}$ and $X_\infty^h := \{h \in X^h \, : \, \inf(h)= -\infty \}$.
	
	Having two boundaries, it is useful to interpret how they relate to one another. In particular we will see that $X_\infty^h$ consists of horofunctions arising as limits of sequences $(h_{x_n})$ where the sequence $(x_n)$ is Gromov. With this in mind, we consider a map $\phi: X_\infty^h \to \partial X$, known as the local minimum map that sends an horofunction to the equivalence class of Gromov sequences that gives rise to it. The major properties of the local minimum map can be found in \cite{maher2018random}, namely it is surjective, continuous and $G$-equivariant.
	
	\begin{definition}{(Basic Assumption)}
		We say a hyperbolic space $X$ satisfies the basic assumption (BA) if the local minimum map is a homeomorphism.
	\end{definition}

	In other words, in hyperbolic spaces satisfying (BA) we are allowed to work interchangeably between the two boundaries. We are mostly interested in $X^h$ for compacity arguments and $\partial X$ for metric ones. Henceforth all our hyperbolic spaces satisfy (BA). This includes strong hyperbolic spaces such as CAT(-1) spaces and hyperbolic groups with a Green metric (see \cite{nica2016strong}.
	
	Finally we equip $\isom(X)$ with a metric that tracks the behaviour of its action in $\bord X$. With effect, given $1<b\leq 2^{1/ \delta}$, take
	\begin{equation*}
		d_G(g_1,g_2) := \max \left\{ \sup_{\xi \in \bord X} D_b(g_1\xi, g_2\xi) \, ; \, \sup_{\xi \in \bord X} D_b(g_1^{-1}\xi, g_2^{-1}\xi) \right\},
	\end{equation*}
	for every $g_1, g_2 \in \isom(X)$. We will prove that $\isom(X)$ is a topological group when equipped with $d_G$, in particular, $d_G$ is a distance. Notice that since $D_b$ is a complete metric, then $\isom(X)$ is complete.
	
	\subsection{Random setting}
	
	Let $\Prob(M)$ and $\Prob_c(M)$ denote, respectively, the space of Borel probability measures and its subspace of with Borel probability measures with compact support in some metric space $M$. Let $G$ be a topological group acting on a metric space $M$, we define the convolution
	\begin{align*}
		\star: \Prob_c(G)\times \Prob(M) &\to \Prob(M) \\
		(\mu , \nu) &\mapsto \mu\star \nu := \int_G g\nu d\mu(g), 
	\end{align*}
	where $g\nu$ is the pushforward of $\nu$ under the action of $g$, in other words, the convolution is the average of the pushforwards with respect to $\mu$. In particular, $G$ acts on itself on the right, this allows for the definition of $\mu^n$ for every $\mu \in \Prob_c(G)$ as the $n-th$ convolution of $\mu$ with itself. Let us stress the fact we are considering $G$ acting on itself through the right action. As a side-note recall that if both $\mu$ and $\nu$ have compact support then so does $\mu \star \nu$.
	
	Working in the degree of generality introduced in the previous section, a problem could arise here. We need $G$ to be second countable in order for the support of a measure in $\Prob(G)$ to be well defined. Whence we restrict our attention to closed separable groups of isometries $G\subset \isom(X)$. Notice that since $\isom(X)$ is complete for the metric $d_G$, so is $G$. This gives for particularly well behaved measures in $G$. With this in mind, henceforward $G$ always stands for a closed separable groups of isometries acting on a hyperbolic space $X$.
	
	Let $\mu \in \Prob_c(G)$ then we will consider the product measure $\mu^\mathbb{N}$ which has compact support in $\Omega = G^\mathbb{N}$. Given $\omega = (g_0, g_1, ..., g_n, ...) \in \Omega$ we set the notation 
	\begin{equation*}
		\omega^n = g_0g_1...g_{n-1},
	\end{equation*} 
	as well as defining the Bernoulli shift  $T:\Omega \to \Omega$  sending $\omega = (g_0, g_1, ... g_n, ...)$ to $T\omega = (g_1,...,g_n,...)$. We will denote by $\omega^{-n}$ the inverse of $\omega^n$, that is, $(\omega^n)^{-1}$.
	
	Given $\mu\in \Prob_c(G)$, by compacity we have
	\begin{equation*}
		\int_G d(gx_0, x_0) d\mu(g) <\infty,
	\end{equation*}
	so we can define the drift
	\begin{equation*}
		\ell(\mu) := \lim_{n\to \infty} \frac{1}{n} \int_G d(gx_0, x_0) d\mu^n(g) =  \lim_{n\to \infty} \frac{1}{n} \int_G d(\omega^nx_0, x_0) d\mu^\mathbb{N}(\omega). 
	\end{equation*}
	Notice that the measurability of the integrand functions follows from continuity. Moreover, due to ergodicity, by Kingman's Ergodic Theorem the limit is $\mu^\mathbb{N}$-almost surely equal to the limit of $d(\omega^n x_0, x_0)/n$.
	
	Our goal in this text is to understand the behaviour of $\ell(\mu)$ with respect to both small perturbations in $\mu$ as well as its convergence rate underlying Kingman's theorem. To understand what we mean by small perturbations in $\mu$ we shall introduce the  Hölder Wasserstein distance whilst for the second we introduce the large deviations estimates.
	
	Let $L^\infty(G)$ stand for the space of Borel measurable functions $\varphi: G\to \mathbb{R}$ bounded in the sup norm. For every $0< \alpha\leq 1$ and $\varphi\in L^\infty(G)$ define the $\alpha$-Hölder constant
	\begin{equation*}
		\upsilon_\alpha^G(\varphi) := \sup_{g, g'\in G, \, g\neq g'} \frac{|\varphi(g)-\varphi(g')|}{d_G(g,g')^\alpha}.
	\end{equation*}
	For every $\mu, \nu\in \Prob_c(G)$, we define the Wasserstein distance between them as
	\begin{equation*}
		W_\alpha(\mu, \nu) = \sup_{\varphi \in L^\infty(G), \, \upsilon_\alpha^G(\varphi)\leq 1} \left| \int_G \varphi d\mu - \int_G \varphi d\nu \right|.
	\end{equation*}
	A detailed discussion on Wasserstein distances can be found in \cite{villani2009optimal}. One now feels tempted to consider the continuity of the drift with respect to the Wasserstein distance. This may be problematic as we could consider very close measures $\mu$ and $\mu_1$ where the support of $\mu_1$ contains elements that expand an arbitrarily large quantity which are picked with an equally small probability. Then the distance between $\mu$ and $\mu_1$ would be small but the drift could be incommensurably different. To work around this problem, given $1< b \leq 2^{1/\delta}$ for every $\lambda>0$ set
	\begin{equation*}
		G_\lambda := \left\{g \in G \, : \, b^{d(gx_0, x_0)}< \lambda \right\}.
	\end{equation*}

	Having defined $G_\lambda$ and the Weisserstein distance, we can describe what we mean by continuity. With effect we want to study the continuity of the map $\ell:\Prob_c(G_\lambda) \to \mathbb{R}$ given by $\mu \mapsto \ell(\mu)$ with respect to the Wasserstein distance. As for large deviations estimates, we are trying to prove a result of the following nature.
	
	\begin{definition}[Large deviations estimates]
		We say that a measure $\mu\in \Prob(G)$ satisfies a large deviation estimate if there are constants $r>0$, $C>0$ and for every small enough $\varepsilon>0$ there exists $n_0 = n_0(\varepsilon)$ such that for every $n>n_0$
		\begin{equation*}
			\mu^\mathbb{N}\left\{\omega \in \Omega \, : \, \left|  \frac{1}{n}d(\omega^nx_0, x_0) - \ell(\mu) \right| > \varepsilon \right\} < e^{-C\varepsilon^2n}.
		\end{equation*}
	\end{definition}

	The drift is related with the intuitive notion of walking towards infinity of the space at a certain speed. This is exactly why we introduced boundaries in the previous section. The following theorem by Karlsson and Gouëzel makes this notion precise by guaranteeing the existence of a horofunction that tracks the process. In other words, there exists a specific direction the trajectory follows towards the boundary.
	
	\begin{theorem}[Karlsson - Gouëzel in \cite{gouezel2020subadditive}]
		\label{Karlsson-Gouëzel}
		Let $\mu\in \Prob_c(G)$, for $\mu^\mathbb{N}$-almost every $\omega \in \Omega$ there exists an horofunction $h_\omega \in X^h$ such that for every $x$
		\begin{equation*}
			\lim_{n\to \infty} \frac{1}{n} h_\omega( \omega^n x) = -\ell(\mu).
		\end{equation*}
		Moreover, if $X$ is separable and $\Omega$ is a standard probability space, one can choose the map $\omega \mapsto h_\omega$ to be Borel measurable.
	\end{theorem}
	
	As a remark, Karlsson and Gouëzel's Theorem is in fact way more general than what we stated, it holds for every cocycle on the group of semicontractions over any ergodic transformation. 
	
	\subsection{Results}	
	
	The first result concerns the aforementioned metric structure of $G$.
	
	\begin{theorem}
		\label{metricGroup}
		Let $X$ be a Gromov hyperbolic space and $\isom(X)$ its group isometries, then $d_G$ is a metric in $\isom(X)$. Moreover $(\isom(X), d_G)$ is a topological group.
	\end{theorem}
	
%
%
	
 	Next we present our a multiplicative ergodic theorem type result for the drift adding information to Theorem \ref{Karlsson-Gouëzel}:
	
	\begin{theorem}[Hyperbolic Multiplicative Ergodic Theorem]
		\label{hmet}
		Let $X$ be a hyperbolic space and $G$ be a closed separable group of isometries of $X$. Given $\mu \in \Prob_c(G)$ with $\ell(\mu)>0$. For $\mu^\mathbb{N}$-almost every $\omega=(g_0,g_1,...,g_n,...)$ in $\Omega$ there is a filtration of the horofunction boundary in subsets
		\begin{equation*}
			X_-^h(\omega) \subset  X_+^h(\omega) = X^h,
		\end{equation*} 
		such that:
		\begin{enumerate}
			\item for every $h\in X_+^h(\omega) \backslash X_-^h(\omega) $
			\begin{equation*}
				\lim_{n\to \infty} \frac{1}{n} h( \omega^n x_0) = \ell(\mu);
			\end{equation*}
			\item for every $h \in X_-^h(\omega)$
			\begin{equation*}
				\lim_{n\to \infty} \frac{1}{n} h( \omega^n x_0) = -\ell(\mu),
			\end{equation*}
			and $X_-^h(\omega)$ consists of a single horofunction.
		\end{enumerate}
		In addition, the filtration is measurable with respect to the completion of the Borel $\sigma$-algebra in $(G,d_G)$. Moreover, the decomposition is $G-$invariant, that is, denoting by $T$ the Bernoulli shift,
		\begin{equation*}
			g_0\cdot X_{-}^h(\omega) = X_{-}^h(T\omega).
		\end{equation*}
	\end{theorem}
	
	\begin{definition}[Irreducible measure]
		We say that $\mu \in \Prob_c(G)$ is irreducible if there is no horofunction $h \in X^h$ such that $g\cdot h = h$ for  $\mu$-almost every $g$.
	\end{definition}
    
    By an argument analogous to Proposition 5.3 in \cite{duarte2016lyapunov}, the set of irreducible measures is open. This allows us to assume that measures in a neighbourhood are also irreducible.
    
	In this text we are interested in understanding the typical behaviour of $\omega^nx_0$. Notice however that $G$ acts on horofunctions by (\ref{horoAction}), this means that when looking at this action in $X$ from the point of view of horofunctions it is more sensible to consider $\omega^{-n}\cdot h$. Hence we take  on $X^h$ the action of $G$ as $(g,h) \to g^{-1}\cdot h$. To keep the local minimum map as a $G$-equivariant map, we also use this action on $\partial X$, that is, $(g,\xi) \to g^{-1} \xi$ .
	
	
	Let once again $M$ be a metric space on which $G$ acts. Given $\mu \in \Prob_c(G)$, a measure $\nu \in \Prob_c(M)$ is $\mu$-stationary if $\mu \star \nu = \nu$. The existence of stationary measures is extremely important to us. It is actually easy to see that when $M=X^h$ the existence of stationary measures follows from compactness. Our technique will actually allow for the existence of stationary measures in $\partial X$, which will require a finer treatment.

	Since $X$ satisfies (BA), the local minimum map $\phi:X_\infty^h \to \partial X$ is a $G$-equivariant homeomorphism. We denote by  $h_\xi$ the horofunction whose image under $\phi$ is $\xi$. We then have the following Furstenberg type formula.
	
	\begin{theorem}[Furstenberg type formula]
		\label{furstenbergForm}
		Let $X$ be a hyperbolic space and $G$ be a closed separable group of isometries of $X$. Let $\mu$ be an irreducible measure in  $\Prob_c(G)$, then for every $\mu-$stationary measure $\nu\in \Prob(\partial X)$
		\begin{equation*}
			\ell(\mu) = \int_G \int_{\partial X} h_\xi(gx_0) d\nu(\xi)d\mu(g).
		\end{equation*}
	\end{theorem}
	
	In §\ref{existUniqStat} we will prove that every measure in a neighbourhood of an irreducible measure with $\ell(\mu)>0$ admits a unique stationary measure. With this information  we will obtain the following result.
	
	\begin{theorem}
		\label{continuity}
		Let $X$ be a hyperbolic space and $G$ be a closed separable group of isometries of $X$. Given $\lambda>0$, let $\mu\in \Prob_c(G_\lambda)$ be irreducible and $\ell(\mu)>0$. Then there are constants $0<\alpha \leq 1$, $C<\infty$ and $r>0$ such that for every $\mu_1, \mu_2\in \Prob_c(G_\lambda)$ if $W_\alpha(\mu_i, \mu)<r$, $i=1,2$, then
		\begin{equation*}
			|\ell(\mu_1)-\ell(\mu_2)| \leq C \, W_\alpha(\mu_1, \mu_2).
		\end{equation*}
	\end{theorem}
	
	\begin{theorem}
		\label{LDT}
		Let $X$ be a hyperbolic space and $G$ be a closed separable group of isometries of $X$. Given $\lambda>0$, let $\mu\in \Prob_c(G_\lambda)$ be irreducible and $\ell(\mu)>0$, there exists $b>1$, $V$ a neighbourhood of $\mu$ in $\Prob_c(G_\lambda)$ and $C=C(\mu)<\infty$, $k=k(\mu)>0$ and $\varepsilon_0 = \varepsilon_0(\mu)>0$ such that for all $0<\varepsilon <\varepsilon_0$, $\mu_1\in V$, $n\in \mathbb{N}$ and every $h\in X_\infty^h$
		\begin{equation*}
			\mu_1^\mathbb{N}\left\{ \omega \in \Omega \, : \,  \left| \frac{1}{n} h(\omega^n x_0)- \ell(\mu_1) \right| > \varepsilon \right\} \leq C b^{-k\varepsilon^2n}.
		\end{equation*}
	\end{theorem}
	
	Applying a comparison argument, we will end up obtaining the intended result:
	
	\begin{corollary}
		\label{LDTdrift}
		Let $X$ be a hyperbolic space and $G$ be a closed separable group of isometries of $X$. Given $\lambda>0$, let $\mu\in \Prob_c(G_\lambda)$ be irreducible and $\ell(\mu)>0$, there exists $b>1$, $V$ a neighbourhood of $\mu$ in $\Prob_c(G_\lambda)$ and $C=C(\mu)<\infty$, $k=k(\mu)>0$ and $\varepsilon_0 = \varepsilon_0(\mu)>0$ such that for all $0<\varepsilon <\varepsilon_0$, $\mu_1\in V$, $n\in \mathbb{N}$,
		\begin{equation*}
			\mu_1^\mathbb{N}\left\{ \omega \in \Omega \, : \,  \left| \frac{1}{n} d(\omega^nx_0\, , \, x_0)- \ell(\mu_1) \right| > \varepsilon \right\} \leq C b^{-k\varepsilon^2n}.
		\end{equation*}
	\end{corollary}
	
	We take this time to recall that if $X$ is proper than all results apply to $G=\isom(X)$.
	
	\subsection{Further Discussion}
	
	The drift in the metric hyperbolic setting has a quite similar behaviour to the Lyapunov exponents in the linear cocycle  setting. In fact, if we consider $SL(2,\mathbb{R})$ acting isometrically on the hyperbolic upper-half plane, the two concepts $\mathbb{H}^2$ overlap. We explore this common playground to use techniques that were originally obtained for the study of the Lyapunov exponents in $SL(2, \mathbb{R})$, while working around the metric technicalities.
	
	Despite the similarities, there are several interesting natural actions of groups acting by isometries on hyperbolic spaces that escape the linear setting such as rank one semisimple Lie Groups acting on their symmetric spaces, Gromov hyperbolic groups acting on their Cayley graphs, mapping class groups on their curve complexes, the Cremona group acting on the Picard-Manin hyperbolic space among others (see \cite{maher2018random}). In some of these examples the fact we drop the usual properness condition is quite important.
	
	When it comes to large deviations, our work meets the paper of Boulanger et. al. in \cite{boulanger2020large} where a large deviations principle is proven in the setting of non-elementary measures in a countable group $G$. More recently Aoun and Sert \cite{aoun2021random} obtained large deviation estimates as well as continuity in the proper case for cocompact actions of $G$ in $X$. Our large deviations are weaker as they local instead of global, although they apply more generally. The local nature of our results are a consequence of methods applied. More precisely, we will use spectral techniques motivated by Duarte-Klein \cite{duarte2016lyapunov}. Such methods have also been used in the case of hyperbolic groups by Björklund in \cite{bjorklund2010central} where a central limit theorem is presented.

	As for continuity, if $\mu$ is non-elementary and supported on a finite set, the drift is known to be analytic with respect to measures supported in the same set (see \cite{gilch2013regularity, gouezel2015analyticity, gouezel2018entropy, ledrappier2013regularity}). As mentioned before, in the proper case our continuity results are also similar to the ones in \cite{aoun2021random}. Note however that our work allows Hölder continuity in compactly supported measures, which makes it so we can consider more general measures and also take different isometries as well as present more information regarding the continuity.
		
	Although we define irreducible measures,  the case where $\mu$ is  non-elementary is also interesting to us. In \cite{maher2018random}, it was proven that in this case the drift is positive; however, non-elementary measures are always irreducible by definition, in other words, our results apply in such situations. 
		
	Non-elementary measures are irreducible, but it may be interesting to explore whether being irreducible is a generic condition. One can easily prove that being irreducible is an open condition in $\Prob_c(G)$. Density seems a bit more delicate since in some cases, such as $X=\mathbb{R}$, there are no irreducible measures. However, due to the convexity of the space of probabilities, irreducible measures, if they exist, form a dense set.
	
	The Fürstenberg type formula is not entirely new either. For hyperbolic spaces one can find it used in  \cite{aoun2021random} whilst for hyperbolic groups it makes appearances in texts such as \cite{bjorklund2010central, gouezel2015analyticity}. More broad versions also exist for arbitrary groups acting on some metric space such as \cite{carrasco2017furstenberg, karlsson2011noncommutative}.
	
	Concerning the paper's organization, in section two we explore Hyperbolic metric spaces and their properties. In section three we take a closer look at the group of isometries $G$, proving it is a topological group and how the Wasserstein distance behaves. We dedicated section four to proving the HMET. Sections five and six are devoted to proving continuity and large deviations respectively. We prove the Furstenberg formula in the end, way after its use since it employs machinery developed all throughout the text.
	
	\section{Horofunction compactification of Hyperbolic Metric Spaces}
	\label{Boundaries}
	
	Throughout this section we discuss the Horofunction compactification of Gromov hyperbolic spaces. Although we repeat some of the definitions that appeared in the introduction we add new flavour and details to them. More profound presentations can be found in \cite{bridson2013metric, das2017geometry}.

	\subsection{Horofunction Compactification}
	
	Let $X$ be a metric space with basepoint $x_0$. Consider the injective map
	\begin{align*}
		\rho: X & \to C(X) \\
		x & \mapsto h_x( \cdot ) = d(\cdot , x) - d(x,x_0).
	\end{align*}
	Throughout the text we will also make use of the forms
	\begin{align*}
		h_x(z) & = d(z, x_0) - 2 \langle z \, , \, x \rangle_{x_0}\\ 
		& = \langle x \, , \, x_0 \rangle_z - \langle x \, , \, z \rangle_{x_0}.
	\end{align*}
	Notice that $h_x$ are all $1$-Lipschitz and satisfy $h_x(x_0)=0$.
	
	Endow the space $C(X) \subset \mathbb{R}^X$ with the product topology, that is the topology of pointwise convergence, which is equivalent to the compact-open topology. Then, using the triangle inequality one has
	\begin{equation*}
		-d(z,x_0) \leq h_x(z) \leq d(z,x_0),
	\end{equation*} 
	hence $\rho(X)$ may be identified with a subset of $\Pi_{z\in X}[-d(z,x_0) , d(z,x_0)]$ which, by Tychonoff's theorem, is compact for the product topology. Therefore the closure $X^h = \overbar{\rho(X)}$ will be a compact set called the horofunction Compactifications of $X$. The elements in $X^h$ are called horofunctions of $X$.
	
	\begin{proposition}
		The horofunction compactification is compact, Hausdorff and second countable (hence metrizable)
	\end{proposition}
	
	\begin{proof}
		By hypothesis, $X$ is a separable metric space, hence Hausdorff and second countable. Since $\mathbb{R}$ is also Hausdorff and second countable, so is $C(X)$ (for the compact-open topology). However for the subspace of $1-$Lipschitz functions normalized by taking the value $0$ at $x_0$ the compact-open topology and the topology of pointwise convergence agree.
	\end{proof}
	
	If $X$ is a proper space then the pointwise convergence coincides with uniform convergence on compact sets from the usual construction of the horofunction compactification, in this case $X^h$ contains $X$ as an open and dense set. In the nonproper case, the image $\rho(X)$ may not be open in $\overbar{\rho(X)}$, so we may not have a compactification in the usual sense.

	\begin{proposition}
		Let $\isom(X)$ be a group of isometries of $X$. Then the action of $\isom(X)$ in $X$ extends to an action by homeomorphisms on $X^h$, defined by
		\begin{equation*}
			g \cdot h(z) :=	h(g^{-1}z) - h(g^{-1}x_0),
		\end{equation*}
		for $g\in \isom(X)$ and $h\in X^h$.
	\end{proposition}
	
	\begin{proof}
		We naturally transport the action on $X$ to the action on $\rho(X)$ via $g\cdot h_x = h_{gx}$. Then
		\begin{align*}
			g\cdot h_x (z) & = h_{gx}(z) \\
			& = d(z,gx) - d(gx,x_0)\\
			& = d(g^{-1}z\, , \, x ) - d(x \,, \, g^{-1}x_0) \\
			& = h_x(g^{-1}z) - h_x(g^{-1}x_0),
		\end{align*}
		which we can transport to the whole of $X^h$. It is immediate that if $h_n\to h$ pointwisely, then $g\cdot h_n \to g\cdot h$ pointwisely, whence the action of $\isom(X)$ on $X^h$ is continuous.
	\end{proof}
	
	As we've previously mentioned, we shall partition horofunctions $X^h$ in two: its finite part 
	\begin{equation*}
		X_F^h:=\{h\in X^h \, : \, \inf(h)> -\infty\}
	\end{equation*}
	and its infinite part
	\begin{equation*}
		X_\infty^h:=\{h\in X^h \, : \, \inf(h)=-\infty\}
	\end{equation*}
	Both $X_F^h$ and $X_\infty^h$ are invariant for the action of $\isom(X)$ on $X^h$. Clearly one has $\rho(X)\subset X_F^h$ and, in well behaved cases, one may actually get the equality. Another important remark to make concerns the fact that $X_\infty^h$ need not be compact. Let us look at a pathologic example that also explains the nomenclature boundary instead of compactification when referring to $\partial X$
	
	\begin{example}
		Consider $X \subset \mathbb{R}^2$ a set given by countably many half-lines emanating from the origin in $\mathbb{R}^2$. Given $x,y \in X$, consider the distance
		\begin{equation*}
			d(x,y) = 
			\begin{cases}
				||x-y|| &\textrm{, if $x$ and $y$ belong to the same half-line} \\
				||x||+||y|| &\textrm{, otherwise}.
			\end{cases}
		\end{equation*}
		This space is a tree, hence it is $0$-hyperbolic. Notice however that the Gromov boundary of $X$ is $\mathbb{N}$ with the discrete topology, which isn't compact. This is a consequence of the fact that locally compactness fails at the origin. Consider now the sequence of horofunctions $h_n$ given as a limit of sequences $h_{x_k}$, where $x_k$ is the point at distance $k$ from the origin in the $n$-th ordered half-line. Let $y$ be a point in $x_n$, then $h_n(y)=-||y||$ whose infimum is $-\infty$, that is, $h_n\in X_\infty^h$.  However that $h_n\to h_{x_0}$ which belongs to $X_F^h$, hence $X_\infty^h$ is not compact.
	\end{example}
	
	In this text we assumed that $X$ satisfies the so called basic assumption (BA), that is, there exists a homeomorphism between $X_\infty^h$ and $\partial X$. For the sake of completeness we now replicate its construction. In particular, this presentation lays bare the interplay between horofunctions and the Gromov product. The following Lemma is an adaptation of a result in \cite{maher2018random}.
	
	\begin{lemma}
		\label{horoProduct1}
		Let $X$ be a $\delta$-hyperbolic space with basepoint $x_0$. Then for every horofunction $h \in X^h$ and points $x,y\in X$, the following inequalities holds:
		\begin{equation*}
			\langle x \, , \, y \rangle_{x_0} \geq \min\{ -h(x) \, , \, -h(y)\} - \delta,
		\end{equation*} 
		moreover, for every $z\in X$,
		\begin{equation*}
			\langle x \, , \, y \rangle_{x_0} \geq \min\{ -h_x(z) \, , \, -h_y(z)\} - \delta.
		\end{equation*}
	\end{lemma}
	
	\begin{proof}
		Let $z\in X$. Using the triangle inequality one has
		\begin{align}
			\langle x \, , \, z \rangle_{x_0} 
			& = \frac{1}{2} (d(x_0, x) + d(x_0, z) - d(x,z)) \nonumber \\
			& \geq d(x_0, z) - d(x,z) \nonumber \\
			& = - h_z(x). \label{gromovHoroIneq}
		\end{align}
		Now, from the definition of hyperbolicity
		\begin{align*}
			\langle x \, , \, y \rangle_{x_0} 
			& \geq \min\{ \langle x \, , \, z \rangle_{x_0} \, 	, \, \langle z \, , \, y \rangle_{x_0} \} - \delta \\
			& \geq \min \{ - h_z(x) \, , \, - h_z(y) \} - \delta.
		\end{align*}
		The claim follows from the fact that every horofunction is the pointwise limit of functions of the form $h_z$.
		
		The second inequality is analogous using $\langle x \, , \, z \rangle_{x_0} \geq -h_x(z)$.
	\end{proof}
	
	\begin{lemma}
		\label{horoProduct2}
		Let $h\in X_\infty^h$ be an horofunction and $(x_n)$ a sequence such that $h_{x_n} \to h$ and $(y_n)$ a sequence such that $h(y_n) \to - \infty$ . Then the sequences $(x_n)$ and $(y_n)$ are Gromov. 
	\end{lemma}
	
	\begin{proof}
		Using the first inequality in Lemma \ref{horoProduct1},
		\begin{equation*}
			\lim_{n,m\to \infty} \langle y_n \, , \, y_m \rangle_{x_0} \geq \lim_{n,m\to \infty} \min\{ -h(y_n) \, , \, -h(y_m)\} - \delta = + \infty,
		\end{equation*}
		hence $(y_n)$ is Gromov.
		
		Using Lemma \ref{horoProduct1} again, for every $z\in X$
		\begin{equation*}
			\langle x_n \, , \, x_m \rangle_{x_0} \geq \min \{ -h_{x_n}(z) , -h_{x_m}(z) \} - \delta.
		\end{equation*}
		Taking the limit as $m,n$ go towards infinity yields
		\begin{equation*}
			\lim_{n,m\to \infty} \langle x_n \, , \, x_m \rangle_{x_0} \geq - \inf_{z\in X} h(z) - \delta =  + \infty.
		\end{equation*}

	\end{proof}
	
	\begin{proposition}
		\label{horoProduct3}
		Let $h\in X_\infty^h$ be an horofunction. Let $(x_n)$ and $(y_n)$ be Gromov sequences such that $h_{x_n}\to h$ and $h(y_n)\to -\infty$, respectively. Then $(x_n)$ and $(y_n)$ are Gromov sequences and  $(x_n)\sim (y_n)$, in particular all sequences $(y_n)$ such that $h(y_n)\to -\infty$ converge to the same boundary point. 
	\end{proposition}
	
	\begin{proof}
		Using Gromov's inequality together with Lemma \ref{horoProduct2} and (\ref{gromovHoroIneq})
		\begin{align*}
			\langle x_n \, , \, y_n \rangle_{x_0} & \geq \min \left\{ \langle x_n \, , \, x_m \rangle_{x_0} , \langle x_m \, , \, y_n \rangle_{x_0} \right\} - \delta \\
			& \geq \min \left\{ \langle x_n \, , \, x_m \rangle_{x_0} , - h_{x_m}(y_n) \right\} - \delta.
		\end{align*}
		Taking the iterated limits towards infinity, we obtain
		\begin{align*}
			\lim_{n\to \infty} \langle x_n \, , \, y_n \rangle_{x_0} & \geq \lim_{n\to \infty} \liminf_{m \to \infty} \min \left\{ \langle x_n \, , \, x_m \rangle_{x_0} , - h_{x_m}(y_n) \right\} - \delta \\
			& = \min \left\{ \lim_{n\to \infty} \liminf_{m \to \infty} \langle x_n \, , \, x_m \rangle_{x_0} , \lim_{n\to \infty} -h(y_n) \right\} - \delta \\
			& = + \infty.
		\end{align*}
	\end{proof}
	
	%
	%
	%
	%
	
	Proposition \ref{horoProduct3} motivates the following definition:
	
	\begin{definition}[Local Minimum Map, in \cite{maher2018random}]
		Define the local minimum map $\phi: X_\infty^h \to \partial X$ given by 
		\begin{equation*}
			\phi(h) = \lim_{n\to \infty} y_n =  \xi,
		\end{equation*}
		where $(y_n)\in \xi$ is such that $h(y_n)\to -\infty$. 
	\end{definition}
	
	The local minimum map is $\isom(X)$-equivariant, continuous and surjective. Proofs for these properties of the local minimum map can be found in \cite{maher2018random}. Our spaces satisfy (BA), in other words, we assume that $\phi$ is also a homeomorphism.
	

%
	
	In the following two lemmas we explore the continuity of the Gromov product. With effect we understand its behaviour upon considering Gromov sequences as arguments.
	
	\begin{lemma}[in \cite{das2017geometry}]
		\label{horobuse1}
		Let $(x_n)$ and $(y_n)$ be two Gromov sequences in a $\delta$-hyperbolic space and fix $y,z\in X$. Then
		\begin{align*}
			\limsup_{n\to \infty} \, \langle x_n \, , \, y \rangle_z 
			& \leq \liminf_{n\to \infty} \, \langle x_n \, , \, y \rangle_z + \delta \\ 
			\limsup_{n,m\to \infty} \, \langle x_n \, , \, y_m \rangle_z 
			& \leq \liminf_{n,m\to \infty} \, \langle x_n \, , \, y_m \rangle_z +2\delta
		\end{align*}
	\end{lemma}
	
	\begin{proof}
		Fix $n_1, n_2 \in \mathbb{N}$. By Gromov's inequality
		\begin{equation*}
			\langle x_{n_1}\, , \, y \rangle_z \geq \min \left\{ \langle x_{n_1}\, , \, x_{n_2} \rangle_z \, , \, \langle x_{n_2}\, , \, y \rangle_z  \right\} - \delta.
		\end{equation*}
		Taking the $\liminf$ over $n_1$ and the $\limsup$ over $n_2$ gives
		\begin{align*}
			\liminf_{n,m\to \infty} \langle x_{n}\, , \, y \rangle_z & \geq \min \left\{ \liminf_{n_1,n_2\to \infty} \, \langle x_{n_1}\, , \, x_{n_2} \rangle_z \, , \, \limsup_{n_2 \to \infty} \, \langle x_{n_2} \, , \, y \rangle_z  \right\} - \delta \\
			& = \limsup_{n\to \infty}\langle x_{n}\, , \, y \rangle_z - \delta.
		\end{align*}
		Where the last equality comes from $(x_n)$ being a Gromov sequence.
		
		The second inequality is analogous using the following inequality which is immediate from iterating the 4-point condition of hyperbolicity:
		\begin{equation*}
			\label{5pcondition}
			\langle x \, , \, w \rangle_u \geq \min \{ \langle x \, , \, y \rangle_u \, , \, \langle y \, , \, z \rangle_u \, , \, \langle z \, , \, w \rangle_u \} - 2\delta.
		\end{equation*}
		
	\end{proof}
	
	\begin{lemma}[in \cite{das2017geometry}]
		\label{horobuse2}
		Fix $\xi, \, \eta \in \partial X$ and $y,z \in X$. For all $(x_n)\in \xi$ and $(y_n)\in \eta$, we have
		\begin{align*}
			\langle \xi \, , \, y \rangle_z - \delta 
			\leq  \liminf_{n \to \infty} \, \langle x_n \, , \, y \rangle_z 
			& \leq \limsup_{n \to \infty} \, \langle x_n \, , \, y \rangle_z 
			\leq \langle \xi \, , \, y \rangle_z + \delta,\\
			\langle \xi \, , \, \eta \rangle_z - 2\delta \leq \liminf_{n,m\to \infty} \, \langle x_n \, , \, y_m \rangle_z 
			&\leq \limsup_{n,m\to \infty} \, \langle x_n \, , \, y_m \rangle_z 
			\leq \langle \xi \, , \, \eta \rangle_z + 2\delta.
		\end{align*}
	\end{lemma}
	
	\begin{proof}
		The two leftmost inequalities are trivial. Suppose that we are given two sequences $(x_n^1), \, (x_n^2) \in \xi$, let
		\begin{equation*}
			x_n =
			\begin{cases}
				x_{n/2}^1 &\textrm{, if n is even} \\
				x_{(n+1)/2}^2 &, \textrm{ if n is odd}.
			\end{cases}
		\end{equation*}
		By the 4-point condition, for every $n$,
		\begin{equation*}
			\langle x_n \, , \, x_n^1 \rangle_{x_0} \geq \min \left\{ \langle x_n \, , \, x_n^2 \rangle_{x_0} \, , \,  \langle x_n^2 \, , \, x_n^1 \rangle_{x_0} \right\} - \delta,
		\end{equation*}
		which yields $(x_n)\in \xi$. Applying the previous lemma to $x_n$ implies
		\begin{equation*}
			\min_{i=1,2} \, \limsup_{n \to \infty} \, \langle x_n^i \, , \, y \rangle_z \leq \max_{i=1,2} \, \liminf_{n \to \infty} \, \langle x_n^i \, , \, y \rangle_z + \delta.
		\end{equation*}
		Now
		\begin{align*}
			\liminf_{n \to \infty} \, \langle x_n^2 \, , \, y \rangle_z -\delta & \leq \limsup_{n \to \infty} \, \langle x_n^2 \, , \, y \rangle_z -\delta \\
			& \leq \liminf_{n \to \infty} \, \langle x_n^1 \, , \, y \rangle_z \\
			& \leq \limsup_{n \to \infty} \, \langle x_n^1 \, , \, y \rangle_z  \\
			& \leq \liminf_{n \to \infty} \, \langle x_n^2 \, , \,y \rangle_z + \delta.
		\end{align*}
		Taking the $\inf$ over all $(x_n^2)\in \xi$ one obtains the statement.
	\end{proof}
%
%
%
%
%
	
	Recall our notations $h_\xi$ whenever $\phi(h_\xi)=\xi$ for every $\xi \in \partial X$. This extends the notation $h_x$ with $x\in X$ given by the map $\rho$ at the beginning of the section. Therefore, given $\xi\in \bord X$ we can refer to $h_\xi$ without confusion.
	
	\begin{lemma}
		\label{comparisonLemma}
		Let $\xi \neq \eta \in \partial X$. Then, for every  $g\in \isom(X)$, there exists a constant $K(\delta, \xi, \eta )$ depending on the hyperbolicity constant $\delta$ and the points $\xi$ and $\eta$ such that
		\begin{equation*}
			\max_{i=\xi,\eta} h_i(gx_0) \leq d(gx_0,x_0) \leq \max_{i=\xi,\eta} h_i(gx_0) + K(\delta, \xi, \eta ).
		\end{equation*}
	\end{lemma}

	\begin{proof}
		For every $g\in G$
		\begin{equation*}
			d(gx_0, x_0)= h_{y_m^i}(gx_0) + 2\langle y_m^i \, , \, gx_0 \rangle_{x_0}.
		\end{equation*}
		Let $(y_m^i)$ be Gromov sequences such that $h_{y_m^i} \to h_i$ for $i=\xi,\eta$
		whence, using the Gromov inequality
		\begin{align*}
			d(gx_0,x_0) & = \max_{i=\xi,\eta} h_{y_m^i}(gx_0) + 2\min_{i=\xi,\eta} \langle y_m^i \, , \, gx_0 \rangle_{x_0} \\
			& \leq \max_{i=\xi,\eta} h_{y_m^i}(gx_0) + 2\langle y_m^1\, , \, y_m^2 \rangle_{x_0} + 2\delta
		\end{align*}
		By Lemma \ref{horoProduct2} and Proposition \ref{horoProduct3} we know $(y_m^1)$ and $(y_m^2)$ are not equivalent, so taking the inferior limit in $m$ one obtains
		\begin{equation*}
			\max_{i=\xi,\eta} h_i(gx_0) \leq d(gx_0,x_0) \leq \max_{i=\xi,\eta} h_i(gx_0) + K(\delta, \xi, \eta),
		\end{equation*}
		for a constant $K(\delta, \xi, \eta) = 2\langle \xi\, , \, \eta \rangle_{x_0} + 4\delta$.
	\end{proof}
	
	\subsection{The Visual Metric}
	
	Let $X$ be a $\delta$-hyperbolic space. An important point about $\bord X = X\cup \partial X$ is its metrizability. We've presented an explicit metric in the introduction, let us briefly recall it. Given $1 < b \leq 2^\frac{1}{\delta}$ consider the symmetric map $\rho_b: \bord X \times \bord X \to \mathbb{R}$ given by
	\begin{equation*}
		\rho_b(\xi\, , \, \eta) = b^{-\langle \xi \, , \, \eta \rangle_{x_0}}.
	\end{equation*} 
	Then the map $\bar{D}_b: \bord X \times \bord X\to \mathbb{R}$ given by
	\begin{equation*}
		\bar{D}_b(\xi \, , \, \eta) = \inf \sum_{i=0}^{n-1} \rho_b(\xi_i\, , \, \xi_{i+1}) 
	\end{equation*}
	where the infimum is taken over finite sequences of points $\xi_i$ such that $\xi_0= \xi$ and $\xi_n = \eta$, satisfies the triangle inequality and the following visual condition
	\begin{equation*}
		\rho_b(\xi\, , \, \eta)/4 \leq 	\bar{D}_b(\xi \, , \, \eta) \leq \rho_b(\xi\, , \, \eta) \textrm{ for every } \xi, \eta \in \partial X.
	\end{equation*}
	Finally, define the metric $D_b: \bord X \times \bord X\to \mathbb{R}$
	\begin{equation*}
		D_b(\xi, \eta) := \min \left\{ \log(b)d(\xi \, , \, \eta) \, ; \,  \bar{D}_b(\xi \, , \, \eta) \right\}.
	\end{equation*}

	Our true goal in this section is to prove the following proposition relating the action of $g$ in $\bord X$ with its action on the horofunction compactification $X^h$
	
	\begin{proposition}
		\label{VisualMetric}
		Let $g\in \isom(X)$ and $\xi, \eta \in \bord X$, then
		\begin{equation*}
			\frac{1}{C(\delta)} b^{- \frac{1}{2} \left[ h_\xi(g^{-1}x_0) + h_\eta(g^{-1}x_0) \right]} 
			\leq \frac{\bar{D}_b(g\xi\, , \, g\eta)}{\bar{D}_b(\xi\, , \, \eta)} 
			\leq  C(\delta)b^{- \frac{1}{2} \left[ h_\xi(g^{-1}x_0) + h_\eta(g^{-1}x_0) \right]},
		\end{equation*}
		where $C(\delta)=4b^{6\delta}$.
	\end{proposition}
	
	\begin{proof}
		We start by using the visual condition
		\begin{equation*}
			\frac{1}{4} \frac{\rho_b(g\xi\, , \, g\eta)}{\rho_b(\xi\, , \, \eta)} 
			\leq \frac{D_b(g\xi\, , \, g\eta)}{D_b(\xi\, , \, \eta)}
			\leq 4 \frac{\rho_b(g\xi\, , \, g\eta)}{\rho_b(\xi\, , \, \eta)}.
		\end{equation*}
		
		Using the definition of Gromov product and some computations yields
		\begin{equation*}
			\langle x \, , \, y \rangle_z - \langle x \, , \, y \rangle_{x_0} = \frac{1}{2}\big( \langle x \, , \, x_0 \rangle_z - \langle x \, , \, z \rangle_{x_0} + \langle y \, , \,x_0  \rangle_z - \langle y \, , \, z \rangle_{x_0} \big).
		\end{equation*}
		Take $(x_n)\in \xi$ and $(y_n)\in \eta$. Substituting in the equality above $x$ by $x_n$, $y$ by $y_n$ and $z=g^{-1}x_0$ and taking limits, by Lemmas \ref{horobuse1} and \ref{horobuse2} we obtain
		\begin{equation*}
			\frac{1}{2}\big( h_\xi(g^{-1}x_0) +  h_\eta(g^{-1}x_0) \big) -6\delta 
			\leq \langle \xi \, , \, \eta \rangle_{g^{-1}x_0} - \langle \xi \, , \, \eta \rangle_{x_0} 
			\leq \frac{1}{2}\big( h_\xi(g^{-1}x_0) +  h_\eta(g^{-1}x_0) \big) + 6\delta.
		\end{equation*} 
		Finally one has
		\begin{equation*}
			\frac{\rho_b(g\xi\, , \, g\eta)}{\rho_b(\xi\, , \, \eta)} = b^{-\left(\langle g\xi \, , \, g\eta \rangle_{x_0} - \langle \xi \, , \, \eta \rangle_{x_0}\right)} 
			= b^{-\left(\langle \xi \, , \, \eta \rangle_{g^{-1}x_0} - \langle \xi \, , \, \eta \rangle_{x_0}\right)},
		\end{equation*}
		and the result follows.
	\end{proof}
	
	\section{The group of isometries}
	
	In this section we shall prove and introduce the main tools we will use surrounding the group of isometries $\isom(X)$ of an hyperbolic space $X$.  Namely, we prove Theorem \ref{metricGroup} as well as the behaviour of the Wasserstein distance with respect to convolution. 
	
	\subsection{$\isom(X)$ as a topological group}

	\begin{proof}[Proof of Theorem \ref{metricGroup}]
		Since $D_b$ is a metric, we are left with proving that $d_G(g_1, g_2)=0$ implies that $g_1=g_2$. Suppose $d_G(g_1, g_2)=0$ and let $x\in X$. Notice $g_1x$ and $g_2x$ are both in $X$ so 
		\begin{equation*}
			\bar{D}_b(g_1x, g_2x)\geq\rho_b(g_1x, g_2x)/4>0 
		\end{equation*}
		hence $D_b(g_1x, g_2x)=\log(b)d(g_1x, g_2x)$. Therefore $d(g_1x, g_2x)=0$ for every $x\in X$, implying $g_1=g_2$.
		
		All that remains is to see that the map $(g, g') \mapsto g^{-1}g'$ is continuous. This will follow from a series of inequalities. First, for every $(g,g'), (g_1, g_1') \in G\times G$,
		\begin{equation*}
			d_G(g^{-1}g', g_1^{-1}g_1') \leq d_G(g^{-1}g', g^{-1}g_1') + d_G(g^{-1}g_1', g_1^{-1}g_1').
		\end{equation*}
		Clearly $d_G(g^{-1}g_1', g_1^{-1}g_1') \leq d_G(g^{-1}, g_1^{-1}) = d_G(g, g_1)$. Moreover, given $\xi \in \bord X$ we have
		\begin{equation*}
			d(g^{-1}g'\xi, g^{-1}g_1'\xi) = d(g'\xi, g_1' \xi).
		\end{equation*}
		Next use Proposition \ref{VisualMetric} to obtain
		\begin{align*}
			\bar{D}_b(g^{-1}g'\xi, g^{-1}g_1'\xi) &= \frac{\bar{D}_b(g^{-1}g'\xi, g^{-1}g_1'\xi)}{\bar{D}_b(g'\xi, g_1'\xi)}\bar{D}_b(g'\xi, g_1'\xi) \\
			& \leq C(\delta)b^{-\frac{1}{2}\left(h_1(gx_0)+h_2(gx_0)\right)}\bar{D}_b(g'\xi, g_1'\xi) \\
			& \leq C(\delta)b^{d(gx_0,x_0)}\bar{D}_b(g'\xi, g_1'\xi),
		\end{align*}
		for some horofunction $h_1, h_2\in X^h$. Splitting into the two possible cases we have either
		\begin{equation*}
			D_b(g^{-1}g'\xi, g^{-1}g_1'\xi) = \log(b)d(g^{-1}g'\xi, g^{-1}g_1'\xi) \leq \bar{D}_b(g^{-1}g'\xi, g^{-1}g_1'\xi)
		\end{equation*}
		or
		\begin{equation*}
			D_b(g^{-1}g'\xi, g^{-1}g_1'\xi) = \bar{D}_b(g^{-1}g'\xi, g^{-1}g_1'\xi) \leq \log(b)d(g^{-1}g'\xi, g^{-1}g_1'\xi).
		\end{equation*}
		In either case, the previous controls yield
		\begin{align*}
			D_b(g^{-1}g'\xi, g^{-1}g_1'\xi) &\leq C(\delta)b^{d(gx_0,x_0)} d_G(g', g_1').
		\end{align*}
		
		Taking the supremum over $\xi$ yields
		\begin{equation*}
			d_G(g^{-1}g', g_1^{-1}g_1') \leq d_G(g, g_1) + C(\delta)b^{d(gx_0,x_0)}d_G(g', g_1').
		\end{equation*}
	\end{proof}

	\subsection{Wasserstein Distance and Convolution}
	
	Let now $G\subset \isom(X)$ be a closed separable group. In this section we will explore the interplay between the Wasserstein distance and convolution. But first recall that given $\lambda>0$, we define
	\begin{equation*}
		G_\lambda := \{g \in G \, : \, b^{d(gx_0, x_0)} < \lambda \}.
	\end{equation*}
	
	\begin{proposition}
		Given $\lambda>0$, let $\mu_1, \mu_2, \nu_1, \nu_2 \in \Prob_c(G_\lambda)$, for every $0<\alpha \leq 1$
		\begin{equation*}
			W_\alpha(\mu_1\star \mu_2 , \nu_1 \star \nu_2)\leq  W_\alpha(\mu_1, \nu_1) + C(\delta)^\alpha \lambda^\alpha W_\alpha(\mu_2, \nu_2).
		\end{equation*}
	\end{proposition} 
	
	\begin{proof}
		Some parts of this proof will feel similar to the proof of Theorem \ref{metricGroup}. We also start with an inequality of the type
		\begin{equation*}
			W_\alpha(\mu_1\star \mu_2 , \nu_1 \star \nu_2) \leq W_\alpha(\mu_1\star \mu_2 , \mu_1 \star \nu_2) + W_\alpha(\mu_1\star \nu_2 , \nu_1 \star \nu_2).
		\end{equation*}
		Let $\varphi\in L^\infty(G)$ with $\upsilon_\alpha^G(\varphi)\leq 1$, using the ideas from the previous proof,
		\begin{align*}
			\left|\int_G \varphi(g_1g) d\mu_1(g_1) - \int_G \varphi(g_1g') d\mu_1(g_1) \right| & \leq\int_G  \left| \varphi(g_1g) - \varphi(g_1g')\right| d\mu_1(g_1) \\
			& \leq \int_G d_G(g_1g, g_1g')^\alpha d\mu_1(g_1)\\
			&  \leq C(\delta)^\alpha \lambda^\alpha d_G(g,g')^\alpha.
		\end{align*}
		In other words the map $g \mapsto \int_G \varphi(g_1g) d\mu_1(g_1)$ is $\alpha$-Hölder with constant $\leq C(\delta)^\alpha \lambda^\alpha d_G(g,g')^\alpha$. Hence
		\begin{equation*}
			\left|\int_G \int_G \varphi(g_1g_2) d\mu_1(g_1)d\mu_2(g_2)- \int_G \int_G \varphi(g_1g_2) d\mu_1(g_1)d\nu_2(g_2) \right| \leq C(\delta)^\alpha \lambda^\alpha W_\alpha(\mu_2, \nu_2).
		\end{equation*}
		Transporting the inequalities from the proof of Theorem \ref{metricGroup} once again, we obtain
		\begin{align*}
			|\varphi(gg_2) - \varphi(g'g_2)|\leq d_G(gg_2, g'g_2)^\alpha \leq  d_G(g, g')^\alpha,
		\end{align*}
		hence
		\begin{align*}
			\bigg|\int_G \int_G \varphi(g_1g_2) d\mu_1(g_1)d\nu_2(g_2) &- \int_G \int_G \varphi(g_1g_2) d\nu_1(g_1)d\nu_2(g_2) \bigg| \leq \\
			& \leq \int_G \left|\int_G \varphi(g_1g_2) d\mu_1(g_1) - \int_G \varphi(g_1g_2) d\nu_1(g_1) \right| d\nu_2(g_2) \\
			& \leq  W_\alpha(\mu_1, \nu_1),
		\end{align*}
		taking the supremums over $\varphi$ in the conditions above yields the result.
	\end{proof}

	\begin{corollary}
		Given $\lambda>0$, let $\mu, \nu \in \Prob_c(G_\lambda)$. Then $\mu^n \in \Prob_c(G_{\lambda^n})$ and for every $0< \alpha \leq 1$ and $n\in \mathbb{N}$
		\begin{equation*}
			W_\alpha (\mu^n, \nu^n) \leq W_\alpha(\mu, \nu) \sum_{i=0}^{n-1} C(\delta)^{i\alpha}\lambda^{i\alpha}.
		\end{equation*}
	\end{corollary}

	\begin{proof}
		For the first statement, notice that by the triangle inequality, for every $g_1,g_2 \in G_\lambda$
		\begin{equation*}
			b^{d(g_1g_2x_0, x_0)} \leq b^{d(g_1x_0, x_0)+d(g_2x_0, x_0)} = b^{d(g_1x_0, x_0)}b^{d(g_2x_0, x_0)} \leq \lambda^2.
		\end{equation*}
		Direct applications of the previous Proposition yield the second statement as
		\begin{align*}
			W_\alpha (\mu^n, \nu^n) & \leq W_\alpha (\mu^n, \mu \star \nu^{n-1}) + W_\alpha (\mu \star \nu^{n-1}, \nu^n) \\
			& \leq W_\alpha(\mu, \nu) + C(\delta)^\alpha  \lambda^\alpha W_\alpha(\mu^{n-1}, \nu^{n-1}) \\
			& \leq W_\alpha(\mu, \nu) \sum_{i=0}^{n-1} C(\delta)^{i\alpha}\lambda^{i\alpha}.
		\end{align*}
	\end{proof}

	\section{Hyperbolic Multiplicative Ergodic Theorem}
	
	Let $G\subset \isom(X)$ where $X$ stands for a Gromov hyperbolic space with basepoint $x_0$ and $\mu \in \Prob_c(G)$ with $\ell(\mu)>0$. Recall the notation $\Omega = G^\mathbb{N}$ and $\mu^\mathbb{N}$ the product measure $\mu^\mathbb{N}$ as well as the drift of $\mu$
	\begin{equation*}
		\ell(\mu):= \lim_{n\to \infty} \frac{1}{n}\int_\Omega d(\omega^n  x_0, x_0) d\mu^\mathbb{N}(\omega) = \lim_{n\to \infty} \frac{1}{n} \int_G d(gx_0,x_0) d\mu^n(g).
	\end{equation*}

	\begin{proof}[Proof of Theorem \ref{hmet}]
		By Karlsson-Gouëzel's Theorem, for almost every $\omega \in \Omega$ there is a horofunction such that
		\begin{equation*}
			\lim_{n\to \infty}\frac{1}{n}h(\omega^nx_0) = -\ell(\mu).
		\end{equation*}
		For such $\omega \in \Omega$, set
		\begin{equation*}
			X_-^h(\omega)= \left\{ h\in X^h \, : \, \lim_{n\to \infty}\frac{1}{n}h(\omega^nx_0) = -\ell(\mu) \right\}.
		\end{equation*}
		Let $h_\xi, h_\eta\in X_-^h(\omega)$ for some  $\xi, \eta \in \partial X$. Then, by Proposition \ref{VisualMetric}, for $1<b\leq 2^{1/\delta}$ we have
		\begin{equation}
			\label{fundamentalIneq}
			\bar{D}_b(\omega^{-n} \xi \, , \, \omega^{-n} \eta) \geq \frac{1}{4b^{6\delta}} b^{- \frac{1}{2} \left[ h_\xi(\omega^nx_0) + h_\eta(\omega^nx_0)\right]} \bar{D}_b( \xi \, , \, \eta).
		\end{equation}
		Using the fact $\partial X$ is bounded and taking $n$ large enough, we see that $\bar{D}_b( \xi \, , \, \eta)$ must be zero. Clearly, the same argument using (\ref{fundamentalIneq}) shows that there is no other equivalence class of horofunctions for which $\lim \frac{1}{n} h(\omega^n x_0)$ takes a negative value. 
		
		Now 
		\begin{equation*}
			X_+^h(\omega) \backslash X_-^h(\omega) = \left\{ h\in X^h \, : \, \liminf_{n\to \infty}\frac{1}{n}h(\omega^nx_0) \geq 0 \right\}
		\end{equation*}
		Let $h\in X_+^h(\omega) \backslash X_-^h(\omega)$ and $h_1\in X_-^h(\omega)$, using (\ref{fundamentalIneq}) again together with the fact $D_b$ is bounded from above by $1$, we obtain that for every $n\in \mathbb{N}$ 
		\begin{equation*}
			C \leq h(\omega^nx_0)  + h_1(\omega^nx_0),
		\end{equation*}
		for some $C\in \mathbb{R}$. However, notice that $|h(\omega^nx_0)| \leq d(\omega^n x_0\, , \, x_0)$, whence
		\begin{equation*}
			C - h_1(\omega^nx_0) \leq h(\omega^nx_0) \leq d(\omega^nx_0\, , \, x_0)
		\end{equation*}
		for every $n\in \mathbb{N}$. Dividing both sides by $n$ and taking limits one has
		\begin{equation*}
			\ell(\mu) = \lim_{n\to \infty} -\frac{1}{n}h_1(\omega^nx_0) \leq \lim_{n\to \infty} \frac{1}{n}h(\omega^nx_0) \leq \ell(\mu),
		\end{equation*}
		which proves the statement.
		
		If in $G$ we consider the completion of the Borel $\sigma$-measure then, since $\mu$ has compact support, $(G, \mu)$ is a standard probability space. By Karlsson-Gouëzel Theorem the sets $X_-^h(\omega)$ are measurable.
		
		For the $G$-invariance of $X_-^h$, first recall we are using the right action. Hence the result follows from the definition of the action of $G$ in $X^h$ as for $\omega = (g_0,g_1,...,g_n,...) \in \Omega$ and $T$ the Bernoulli shift,
		\begin{align*}
			\lim_{n\to \infty} \frac{1}{n} g_0 \cdot h(\omega^nx_0) & = \lim_{n\to \infty} \frac{1}{n}\left( h( g_0^{-1}\omega^n x_0) - h(g_0^{-1}x_0)  \right) \\
			& = \lim_{n\to \infty} \frac{n-1}{n} \frac{1}{n-1}\left( h( (T\omega)^{n-1}x_0) - h(g_0^{-1}x_0)  \right) \\
			& = \lim_{n\to \infty} \frac{1}{n-1}h\left( (T\omega)^{n-1}x_0\right) \\
			& = \lim_{n\to \infty} \frac{1}{n}h\left( (T\omega)^{n}x_0\right),
		\end{align*}
		for every $h \in X^h$, in particular, $g_0\cdot X_-^h(\omega) = X_-^h(T\omega)$.
	\end{proof}	
	
	\section{Continuity of the drift in Random Walks}
	
	In this section we prove continuity of the drift with respect to the measure. The main ingredient in the proof is Furstenberg's formula. With that in mind we  will prove that there is a unique stationary measure in $\partial X$ for the random walk in $G$, which thereafter makes the argument somewhat direct.

	\subsection{Existence and Uniqueness of the  stationary measure}
	\label{existUniqStat}

	For every $ f \in L^\infty(\partial X)$ and $0< \alpha \leq 1$ define
	\begin{align*}
		\upsilon_\alpha(f) & :=  \sup_{ \xi \neq \eta \in \partial X} \frac{|f( \xi)-f(\eta)|}{D_b(\xi, \eta)^\alpha},\\
		||f||_\alpha & := ||f||_\infty + \upsilon_\alpha(f).
	\end{align*}
	Set
	\begin{equation*}
		\mathcal{H}_\alpha(\partial X) : = \left\{\ f\in L^\infty(\partial X) \, : \, ||f||_\alpha < \infty \right\}.
	\end{equation*}
	the space of boundary Hölder continuous functions in $ \partial X$. We call $\upsilon_\alpha(f)$ the Hölder constant of $f$. The space $\mathcal{H}_\alpha(\partial X)$ is Banach algebra with unity $\mathbf{1}$.
	
	Given $\mu\in \Prob_c(G)$, define the Markov operator $Q_\mu :L^p(\partial X) \to L^p(\partial X)$ by 
	\begin{equation*}
		(Q_\mu f)(\xi) :=\int_G f(g^{-1}\xi)d\mu(g),
	\end{equation*}
	for $1\leq p\leq \infty$. A simple computation yields that for every $\nu \in \Prob(\partial X)$ and $f\in L^1(\partial X)$
	\begin{align*}
		\int_{\partial X} (Q_\mu f)(\xi) d\nu(\xi) & = \int_{\partial X} \int_G f(g^{-1}\xi) d\mu(g) d\nu(\xi) \\
		&= \int_{\partial X} f(\xi) d\mu\star\nu(\xi),
	\end{align*}
	which yields the following proposition.
	\begin{proposition}
		\label{statprop}
		Let $\mu \in \Prob_c(G)$, then $\nu \in \Prob(\partial X)$ is $\mu-$stationary if and only if for every $f\in L^1(\partial X)$
		\begin{equation*}
			\int_{\partial X} (Q_\mu f) d\nu = \int_{\partial X} f d\nu
		\end{equation*}
	\end{proposition}
	
	We also have the following identity
	\begin{align*}
		(Q_\mu^n f)(\xi) &= \int_G  f(g^{-1}\xi) d\mu^n(g) \\
		& = \int_G \int_G  f(g_{n-1}^{-1}g^{-1}\xi) d\mu^
		{n-1}(g)d\mu(g_{n-1}) \\
		& = \int_G Q_{\mu^{n-1}}f(g_{n-1}^{-1}\xi) d\mu(g_{n-1}) \\
		& =(Q_{\mu} (Q_{\mu^{n-1}} f)) (\xi),
	\end{align*}
	in other words, for every $n\in \mathbb{N}$, $Q_{\mu^n} = Q_\mu^n$.
	
	Given $\mu \in \Prob_c(G)$ and $0< \alpha< 1$ define the average Hölder constant of $\mu$ as
	\begin{equation*}
		k_\alpha^n(\mu) := \sup_{\xi \neq \eta \in \partial X} \int_G \left( \frac{D_b(g^{-1}\xi \, , \, g^{-1}\eta )}{D_b(\xi\, , \, \eta)} \right)^ \alpha d\mu^n(g).
	\end{equation*}
	\begin{remark}
		Notice that the supremum is taken over $\partial X$ where $D_b= \bar{D}_b$.
	\end{remark}
	The relevance of $k_\alpha^n(\mu)$ becomes evident in the following lemma where we relate it with the contracting behaviour of the Markov operator of $\mu$.
	
	\begin{lemma}
		\label{lemmaContraction}
		For every $f\in \mathcal{H}_\alpha(\partial X)$
		\begin{equation*}
			\upsilon_\alpha(Q_{\mu^n}f) \leq k_\alpha^n(\mu) \upsilon_\alpha(f).
		\end{equation*}
	\end{lemma}
	
	\begin{proof}
		Given $f\in \mathcal{H}_\alpha(\partial X)$ and $\xi \neq \eta$
		\begin{align*}
			|(Q_{\mu^n} f)(\xi)-(Q_{\mu^n} f)(\eta)| & \leq \int_G \left| f(g^{-1}\xi) - f(g^{-1}\eta)  \right| d\mu^n(g) \\
			& \leq \upsilon_\alpha
			(f) \int_G D_b(g^{-1}\xi, g^{-1}\eta)^\alpha d\mu^n(g) \\
			& \leq \upsilon_\alpha(f) k_\alpha^n(\mu) D_b(\xi, \eta)^\alpha.
		\end{align*}
		Passing $D_b(\xi, \eta)^\alpha$ to the left side and then taking the supremum over $\xi \neq \eta \in \partial X$ yields the result.
	\end{proof}
	
	In particular, the previous Lemma implies that the Markov operator restricts to a well defined operator in $Q_\mu : \mathcal{H}_\alpha(\partial X) \to \mathcal{H}_\alpha(\partial X)$. In the following Lemma we prove $k_\alpha^n(\mu)$ is submultiplicative, which emphasises the spectral character of the measurement $k_n^\alpha$. 
	
	\begin{lemma}
		\label{submultiplicative}
		Let $\mu \in \Prob_c(G)$, for every $m,n \in \mathbb{N}$
		\begin{equation*}
			k_\alpha^{m+n}(\mu) \leq k_\alpha^{m}(\mu)k_\alpha^{n}(\mu).
		\end{equation*}
	\end{lemma}
	
	\begin{proof}
		For every $\xi$, $\eta$ in $\partial X$
		\begin{align*}
			\int_G \bigg( \frac{D_b(g^{-1}\xi \, , \, g^{-1}\eta )}{D_b(\xi\, , \, \eta)} & \bigg)^ \alpha  d\mu^{m+n}(g) \leq \int_{G\times G} \left( \frac{D_b(g_2^{-1}g_1^{-1}\xi \, , \, g_2^{-1}g_1^{-1}\eta )}{D_b(\xi\, , \, \eta)} \right)^ \alpha d\mu^{m}(g_2)d\mu^n(g_1) \\
			& \leq \int_G \left( \frac{D_b(g_2^{-1}g_1^{-1}\xi \, , \, g_2^{-1}g_1^{-1}\eta )}{D_b(g_1^{-1}\xi\, , \, g_1^{-1}\eta)} \right)^ \alpha \left( \frac{D_b(g_1^{-1}\xi \, , \, g_1^{-1}\eta )}{D_b(\xi\, , \, \eta)} \right)^ \alpha d\mu^{m}(g_2)d\mu^n(g_1) \\
			& \leq \int_G \left( \frac{D_b(g_1^{-1}\xi \, , \, g_1^{-1}\eta )}{D_b(\xi\, , \, \eta)} \right)^ \alpha d\mu^n(g_1) \sup_{\xi \neq \eta \in \partial X} \int_G \left( \frac{D_b(g_2^{-1}\xi \, , \, g_2^{-1}\eta )}{D_b(\xi\, , \, \eta)} \right)^ \alpha  d\mu^{m}(g_2).
		\end{align*}
		Taking the supremum over $\xi$ and $\eta$ yields the result.
	\end{proof}
	
	Before we proceed with the following Proposition, be wary that the Wasserstein distance can also be defined in $\Prob(\partial X)$ as	
	\begin{equation*}
		W_\alpha(\nu_1, \nu_2) := \sup_{ f\in \mathcal{H}_\alpha(\partial X),\, \upsilon_\alpha(f)\leq 1} \left|\int_{\partial X} fd\nu_1 - \int_{\partial X} fd\nu_2 \right|.
	\end{equation*}
	
	\begin{proposition}
		\label{uniquenessProp}
		Let $\mu \in \Prob_c(G)$. If for some $n\in \mathbb{N}$ and $0<\alpha \leq 1$
		\begin{equation*}
			k_\alpha^n(\mu)^{1/n}  < 1,
		\end{equation*} 
		then there exists a unique $\mu$-stationary measure $\nu \in \Prob(\partial X)$. Moreover, for every $f\in \mathcal{H}_\alpha(\partial X)$, 
		\begin{equation*}
			\lim_{n\to \infty} Q_\mu^n(f) = \left( \int_{\partial X} fd\nu_\mu\right) \mathbf{1}.
		\end{equation*}
		
	\end{proposition}

	\begin{proof}
	    The seminorms $\upsilon_\alpha$ are norms in the space $\mathcal{H}_\alpha(\Gamma) / \mathbb{C}\mathbf{1}$. Since $Q_g^n \mathbf{1}=\mathbf{1}$, by hypothesis, $Q_g^n$ acts in $\mathcal{H}_\alpha(\Gamma) / \mathbb{C}\mathbf{1}$ as a contraction. Using spectral theory (see chapter IX in \cite{riesz2012functional} for example), there exists and invariant space $H_0$, isomorphic to $\mathcal{H}_\alpha(\Gamma) / \mathbb{C}\mathbf{1}$, such that $\mathcal{H}_\alpha(\Gamma) = H_0 \oplus \mathbb{C}\mathbf{1}$. Given $f\in \mathcal{H}_\alpha(\Gamma)$ we may write it as $c\mathbf{1} + h$ where $c\in \mathbb{C}$ and $h\in H_0$. With that in mind, define
		\begin{align*}
			\Lambda : \mathcal{H}_\alpha(\partial X) & \to \mathbb{C} \\
			c\mathbf{1}+h \mapsto c.
		\end{align*}
		Now notice that $Q_\mu$ is a positive operator, therefore so is $\Lambda$ as
		\begin{equation*}
			c\mathbf{1}=\lim_{n\to \infty}\left(c\mathbf{1} + Q_\mu^n(h)\right)= \lim_{n \to \infty} Q_\mu^n(f) \geq 0,
		\end{equation*}
		provided $f\geq 0$. Hence $c=\Lambda(f)\geq 0$. Positivity also implies continuity with respect to the uniform norm as
		\begin{equation*}
			|\Lambda(\varphi)| \leq |\Lambda(||\varphi||_\infty \mathbf{1})|| = ||\varphi||_\infty.
		\end{equation*}
		
		Now since $\partial X$ is a metric space, the set of bounded Lipschitz functions in $\partial X$ is dense in the space of bounded uniformly continuous functions $C_b(\partial X)$. With effect, given $f\in C_b(\partial X)$ one can take the functions
		\begin{equation*}
			f_n(\xi) = \inf_{\eta\in \partial X} \{f(\eta)-nD_b(\xi, \eta)\},
		\end{equation*}
		which are all bounded Lipschitz and uniformly converge to $f$. Since the space is bounded, the set of Lipschitz functions is contained in the space of Hölder functions, so $\mathcal{H}_\alpha(\partial X)$ is dense in $C_b(\partial X)$. Hence, $\Lambda$ extends to a positive linear continuous functional $\hat{\Lambda}:C_b(\partial X) \mapsto \mathbb{C}$. 
		
		Riesz-Kakutani-Markov for non-compact spaces (Theorem 1.3 in \cite{sentilles1972bounded}) applies, so there exists a measure $\nu\in \Prob(\partial X)$ such that $\hat{\Lambda}(f)=\int_{\partial X} f d\nu$ for every $f\in C_b(\partial X)$. Finally, writing $f$ once again as $c\mathbf{1} + h$ yields 
		\begin{equation*}
			\int_{\partial X} Q_\mu f d\nu = \hat{\Lambda}(Q_\mu f) = c = \hat{\Lambda}(f)=\int_{\partial X} f d\nu.
		\end{equation*}
		By yet another density argument, this holds for all $f\in L^1(\partial X)$, therefore $\nu$ is $\mu$-stationary. This density of $C_b(\partial X)$ in $L^1(\partial X)$ also justifies the uniqueness of the measure satisfying $\hat{\Lambda}(f)=\int_{\partial X} f d\nu$.
	\end{proof}

	We will now focus on proving $k_\alpha^n(\mu_1)<1$ in a neighbourhood of $\mu$, provided $\mu$ is irreducible and $\ell(\mu)>0$. 
	
	\begin{lemma}
		\label{unifConv}
		Let $\mu \in \Prob_c(G)$ be irreducible and $\ell(\mu)>0$,  then
		\begin{equation*}
			\lim_{n\to \infty} \frac{1}{n} \int_G  h(gx_0) d\mu^n(g)= \ell(\mu)
		\end{equation*}
		uniformly on $h \in  X_\infty^h$.
	\end{lemma}
	
	\begin{proof}
		Consider the Bernoulli shift $T:\Omega \to \Omega$, where $\Omega = G^\mathbb{N}$. By the HMET (Theorem \ref{hmet}), for almost every $\omega \in \Omega$ there exists a filtration of $  X_-^h(\omega)\subset X_+^h(\omega) = X^h$ such that for every $h\in X_+^h(\omega)$
		\begin{equation}
			\label{posHoro}
			\lim_{n\to \infty} \frac{1}{n} h(\omega^n x_0) = \ell(\mu),
		\end{equation}
		and for every $h \in X_-^h(\omega)$
		\begin{equation*}
			\lim_{n\to \infty} \frac{1}{n} h(\omega^n x_0) = -\ell(\mu).
		\end{equation*}
		
		Consider the set
		\begin{equation*}
			S:=\{ h \in X_\infty^h \, : \, h\in X_-^h(\omega), \, \mu^\mathbb{N} \textrm{-almost surely} \}.
		\end{equation*}
		The existence and measurability of $S$ is discussed in more detail in the last section of the paper. Notice that $g_0\cdot X_-^h(\omega) = X_-^h(T\omega)$ for every $ \omega = (g_0,g_1,...) \in \Omega$, so one obtains that $g_0S = S$ for every $g_0 \in \supp \mu$. Moreover $S$ consists at most of one point. However $\mu$ is irreducible, so $S$ must be empty, hence (\ref{posHoro}) holds for every $h\in X_\infty^h$,  $\mu^\mathbb{N}$-almost surely. Since $h(\omega^n x_0)$ are uniformly bounded by $d(\omega^n x_0, x_0)$, by the dominated convergence theorem one has
		\begin{equation}
			\label{convergence}
			\lim_{n\to \infty} \frac{1}{n}\int_G  h(gx_0) d\mu^n(g) = \ell(\mu)
		\end{equation}
		pointwise in $ h \in X_\infty^h$.
		
		Let us now prove the uniformity in $h$. Using an absurd argument, suppose there is a sequence of horofunctions $(h_n)$ in $X_\infty^h$ and $\varepsilon>0$ such that
		\begin{equation*}
			\lim_{n\to \infty} \frac{1}{n}\int_G  h_n(gx_0) d\mu^n(g) < \ell(\mu)- \varepsilon
		\end{equation*}
		for every large $n$.
		Due to the compactness of $X^h$ we can assume that $h_n$ converges to some $h$. Take $(y_m^n)_m$ a family of Gromov sequences such that $h_{y_m^n} \to h_n$ as $m\to \infty$. In particular each $y_m^n$ converges to some $\xi_n\in \partial X$ such that $h_n = h_{\xi_n}$. Then,
		\begin{align*}
			\lim_{n\to \infty} h_n(\omega^nx_0) - d(\omega^nx_0\, , \, x_0)  & = \lim_{n\to \infty} \lim_{m \to \infty} h_{y_m^n}(\omega^nx_0) - d(\omega^nx_0\, , \, x_0)\\
			& = \lim_{n\to \infty} \lim_{m \to \infty} d(y_m^n\, , \, \omega^nx_0) - d(y_m^n\, , \, x_0) - d(\omega^nx_0\, , \, x_0) \\
			& = \lim_{n\to \infty} \lim_{m \to \infty} - 2\langle y_m^n, \omega^nx_0 \rangle_{x_0} \\
			& \geq \lim_{n \to \infty} - 2\langle \xi_n, \omega^nx_0 \rangle_{x_0} - 2 \delta.
		\end{align*}
		The quantity $\langle \xi_n, \omega^nx_0 \rangle_{x_0}$ goes to infinity if and only if both $\xi_n$ and $\omega^nx_0$ converge to the same point in $\partial X$. This would imply $h\in X_\infty^h$ and, by Proposition \ref{horoProduct3} , $\lim_{n \to \infty}h(\omega^nx_0) = -\infty$, hence $h\in X_-^h(\omega)$. Therefore $\langle \xi_n, \omega^nx_0 \rangle_{x_0}$ must $\mu^\mathbb{N}$-almost surely be finite as otherwise $h\in S=\emptyset$. Using dominated convergence theorem again,
		\begin{align*}
			\lim_{n\to \infty} \frac{1}{n}\int_G h_n(gx_0)d\mu^n(g) & = \lim_{n\to \infty} \frac{1}{n}\int_G  d(gx_0\, , \, x_0)d\mu^n(g) + \lim_{n\to \infty} \frac{1}{n}\int_G  h_n(gx_0)-d(gx_0\, , \, x_0)d\mu^n(g) \\
			& = \ell(\mu) + 0 = \ell(\mu),
		\end{align*}
		which yields the absurd.
	\end{proof}
	
%
%

	Applying Proposition \ref{VisualMetric} one now has the inequality
	\begin{equation*}
		k_\alpha^n(\mu) \leq C(\delta)^\alpha\sup_{h\in X^h} \int_G b^{-\alpha h(g{x_0})} d\mu^n(g).
	\end{equation*}
	Since for every horofunction $h\in X^h$ and $g_1, g_2\in G$
	\begin{align*}
		h(g_1g_2x_0) & = \left(h(g_1g_2x_0) - h(g_1x_0) \right) + \left(h(g_1x_0) - h(x_0) \right)\\
		& = g_1^{-1}\cdot h(g_2x_0) + h(g_1x_0),
	\end{align*} 
	one can easily verify that the process $ \sup_{h\in X^h} \int_G b^{-\alpha h(gx_0)} d\mu^n(g)$ is submultiplicative. This fact is relevant to us since it allows us to pass from the spectral quantity $k_\alpha^n(\mu)$ to a more manageable one at the loss of a multiplicative constant. With that in mind we now take a closer look at the quantity on the right, in particular, next lemma tells us that $g \mapsto b^{-\alpha h(gx_0)}$ is Hölder continuous. 
	
	\begin{lemma}
		\label{isHolder}
		Given $\lambda>0$, for every $g_1, g_2 \in G_\lambda$ one has
		\begin{equation*}
			\left|b^{-\alpha h(g_1x_0)} - b^{-\alpha h(g_2x_0)}\right| \leq \lambda^{2\alpha}d_G(g_1, g_2)^\alpha.
		\end{equation*}
	\end{lemma}

	\begin{proof}
		Start by noticing that $x\mapsto x^\alpha$ is $\alpha$-Hölder with Hölder constant $1$. Then we only need to control $|b^{- h(g_1x_0)} - b^{- h(g_2x_0)}|$, which we can do using the mean value theorem and the fact horofunctions are Lipschitz
		\begin{align*}
			\left|b^{- h(g_1x_0)} - b^{ -h(g_2x_0)}\right| & \leq \log(b)\lambda \left| h(g_1x_0) - h(g_2x_0)\right| \leq \lambda \log(b) d(g_1x_0, g_2x_0).
		\end{align*}
		If $D_b(g_1^{-1}x_0, g_2^{-1}x_0) = \log(b) d(g_1^{-1}x_0, g_2^{-1}x_0)$ we are done, otherwise and immediate computation yields
		\begin{align*}
			\log(b)	d(g_1x_0, g_2x_0) \leq b^{d(g_1x_0, g_2x_0)/2} & \leq b^{\left(d(g_1x_0, x_0)+d(g_2x_0, x_0)\right)/2}b^{-\langle g_1x_0 \, , \, g_2x_0 \rangle_{x_0}}\\
			& \leq \lambda D_b(g_1x_0, g_2x_0)
		\end{align*}
		so we are done.
	\end{proof}
	
	\begin{proposition}
		\label{Contract1}
		Given $\lambda>0$ let $\mu \in \Prob_c(G_\lambda)$ be an irreducible measure with $\ell(\mu)>0$. There are numbers $r>0$, $0<\alpha \leq 1$, $0<k<1$ and $n\in \mathbb{N}$ such that for every $\mu_1 \in \Prob_c(G_\lambda)$ satisfying $W_\alpha(\mu, \mu_1)<r$ one has $\;k_\alpha^n(\mu)\leq k$.
	\end{proposition}
	
	\begin{proof}
		By Lemma \ref{unifConv}, for every $h\in X^h$
		\begin{equation*}
			\limsup_{n \to \infty}
			\frac{1}{n} \int_G  h(gx_0) d\mu^n(g) = \ell(\mu)>0.
		\end{equation*}
		In particular, there exists $n_0\in \mathbb{N}$ large enough such that $\int_G  h(gx_0) d\mu^{n_0}(g) \geq \frac{1}{\log(b)}>0$ for every $h\in X^h$. 
		
		Let $\mu \in \Prob_c(G)$ and $h\in X^h$. Use the inequality 
		\begin{equation*}
			b^x < 1 + \log(b)x + \log(b)^2\frac{x^2}{2} b^{|x|},
		\end{equation*}
		to obtain,
		\begin{align*}
			\int_G  b^{-\alpha h(gx_0)} d\mu^n(g)
			& \leq   1 - \alpha\log(b) \int_G h(gx_0)\mu^n(g) \\
			& \hspace{1cm}+ \log(b)^2\frac{\alpha^2}{2} \int_G h(gx_0)^2 b^{\alpha |h(gx_0)|} d\mu^n(g) \bigg) \\
			& \leq 1 - \alpha + \alpha^2 \log(\lambda)^{2n_0}\lambda^{n_0}/2.
		\end{align*}
		Hence there exists $0< \rho<1$ and  $\alpha$ small enough so that $\int_G  b^{-\alpha h(gx_0)} d\mu^n(g) \leq \rho$. Fix such $\alpha$ and $\rho$ for the remainder of the proof. 
	
		To extend this control to close measures notice that by the previous lemma $g\mapsto b^{-\alpha h(gx_0)}$ is $\alpha$-Hölder with Hölder constant $\lambda^{2\alpha}$ for every $h\in X^h$. So taking $\mu, \mu_1$ with $W_\alpha(\mu, \mu_1)\leq r$, where $r$ is at least smaller than $1$ and chosen later,
		\begin{align*}
			\left| \int_G b^{-\alpha h(gx_0)} \mu^{n_0}(g) - \int_G b^{-\alpha h(gx_0)} \mu_1^{n_0}(g) \right| & 
			\leq \lambda^{2\alpha} W_\alpha (\mu^{n_0}, \mu_1^{n_0}) \\
			& \leq W_\alpha(\mu, \mu_1)\sum_{i=0}^{n-1} C(\delta)^{i\alpha}\lambda^{(i+2)\alpha}.
		\end{align*}
		So we can now choose $r$ small enough to ensure there exists $\rho^* \in (\rho,1)$ 
		\begin{align*}
			\left| \int_G b^{-\alpha h(g^{-1}x_0)} \mu^{n_0}(g) - \int_G b^{-\alpha h(g^{-1}x_0)} \mu_1^{n_0}(g) \right| \leq \rho^* - \rho.
		\end{align*}
		Hence
		\begin{align*}
			\int_G b^{-\alpha h(gx_0)} \mu_1^{n_0}(g) & \leq  \int_G b^{-\alpha h(gx_0)} \mu^{n_0}(g) + \left| \int_G b^{-\alpha h(gx_0)} \mu^{n_0}(g) - \int_G b^{-\alpha h(gx_0)} \mu_1^{n_0}(g) \right| \\
			& \leq \rho^* <1
		\end{align*}
		
		Due to the submultiplicativity, picking $\sigma = (\rho^*)^{\frac{1}{n_0}}$, for every $n \in \mathbb{N}$ there exists a constant $C>0$ such that
		\begin{equation*}
			\sup_{h\in X^h}\int_G  b^{-\alpha h(gx_0)} d\mu^n(g) < C \sigma^n.
		\end{equation*}
		Finally  as observed before we now have
		\begin{equation*}
			k_\alpha^n(\mu_1) \leq  C(\delta)^\alpha\, C \, \sigma^n,
		\end{equation*}
		for every $\mu_1$ with $W_\alpha(\mu, \mu_1)<\delta$. In particular, there exists $n\in \mathbb{N}$ for which this quantity is smaller than $1$. 
	\end{proof}
	
	\subsection{Continuity}
	
	In the previous section we have proven that in a neighbourhood of every irreducible measure in $G$ with positive drift all measures admit  a unique stationary measure in $\partial X$. In the next Lemma we explore how the stationary measures behave under perturbations on the measure in $G$.
	
	\begin{lemma}
		Given $\lambda>0$, let $\mu_1, \mu_2\in \Prob_c(G_\lambda)$ and $\nu_{\mu_1}, \nu_{\mu_2} \in \Prob(\partial X)$ their respective stationary measures. Suppose for some $0< \alpha \leq 1$, $\max \{ k_\alpha^n(\mu_1), k_\alpha^n(\mu_2)\} \leq k <1$ (in particular, $\nu_{\mu_1}$ and $\nu_{\mu_2}$ exist), then for every $n\in  \mathbb{N}$ and $f \in \mathcal{H}_\alpha(\partial X)$
		\begin{equation*}
			\left| \int_{\partial X} f d\nu_{\mu_1} - \int_{\partial X} f d\nu_{\mu_2} \right| \leq \frac{\upsilon_\alpha(f)}{1-k}W_\alpha(\mu_1, \mu_2).
		\end{equation*}
	\end{lemma}

	\begin{proof}
		The Markov operators satisfy
		\begin{align*}
			\left|\left| Q_{\mu_1}f - Q_{\mu_2}f \right|\right|_\infty & \leq \sup_{ \xi \in \partial X} \left| \int_G f(g^{-1}\xi) d\mu_1(g) - \int_G f(g^{-1}\xi) d\mu_2(g) \right| \\
			& \leq \upsilon_\alpha(f) W_\alpha(\mu_1, \mu_2).
		\end{align*}
		For the powers we get
		\begin{align*}
			\left|\left| Q_{\mu_1}^nf - Q_{\mu_2}^nf \right|\right|_\infty & \leq \sum_{i=0}^n ||Q_{\mu_2}^i (Q_{\mu_1} - Q_{\mu_2})(Q_{\mu_1}^{n-i-1}(f))||_\infty \\
			& \leq \sum_{i=0}^n ||(Q_{\mu_1} - Q_{\mu_2})(Q_{\mu_1}^{n-i-1}(f))||_\infty \\ 
			& \leq W_\alpha(\mu_1, \mu_2) \sum_{i=0}^n \upsilon_\alpha(Q_{\mu_1}^{n-i-1}(f)) \\
			& \leq  W_\alpha(\mu_1, \mu_2) \upsilon_\alpha(f) \sum_{i=0}^n k^{n-i-1}\\
			& \leq \frac{\upsilon_\alpha(f)}{1-k}W_\alpha(\mu_1, \mu_2). 
		\end{align*}
		Now $\lim_{n\to \infty} Q_{\mu_1}f = (\int_{\partial X} f d\nu_{\mu_1} )\mathbf{1}$ and $\lim_{n\to \infty} Q_{\mu_2}f = (\int_{\partial X} f d\nu_{\mu_2} )\mathbf{1}$ so
		\begin{equation*}
			\left| \int_{\partial X} f d\nu_{\mu_1} - \int_{\partial X} f d\nu_{\mu_2} \right| \leq \sup_n \left|\left| Q_{\mu_1}^nf - Q_{\mu_2}^nf \right|\right|_\infty,
		\end{equation*}
		from which we obtain the result.
	\end{proof}

	\begin{proof}[Proof of Theorem \ref{continuity}]
		Let $\nu\in \Prob( \partial X)$, for every $g, g' \in G_\lambda$, applying the mean value theorem with $x\mapsto b^{-\alpha x}$ as well as Lemma \ref{isHolder}
		\begin{align*}
			\bigg|\int_{\partial X}h_\xi(gx_0) d\nu(\xi) &- \int_{\partial X}h_\xi(g'x_0)d\nu(\xi)\bigg|  \leq  \int_{\partial X} \left| h_\xi(gx_0) -  h_\xi(g'x_0)\right| d\nu(\xi)\\
			&\leq \frac{\max \{b^{d(gx_0,x_0)}, b^{d(g'x_0,x_0)}\}}{\alpha \log(b)}\int_{\partial X}  |b^{-\alpha h_\xi(gx_0)}- b^{-\alpha h_\xi(g'x_0)}|  d\nu(\xi)\\
			& \leq \frac{\lambda}{\alpha \log(b)}\int_{\partial X}  |b^{-\alpha h_\xi(gx_0)}- b^{-\alpha h_\xi(g'x_0)}|  d\nu(\xi)\\
			& \leq \frac{\lambda^3}{\alpha \log(b)}d_G(g,g')^\alpha,
		\end{align*}
		in particular, the map $g\mapsto \int_{\partial X} h(g^{-1}x_0) d\nu(g)$ is Hölder continuous. 
		
		Let $\mu\in \Prob_c(G_\lambda)$ be irreducible with $\ell(\mu)>0$. Then there exist $0<\alpha \leq 1$ and a neighbourhood of $\mu$ in which all measures satisfy $k_\alpha^n(\mu_1) <1$. Let $\mu_1,\, \mu_2$ be in a neighbourhood of $\mu$ and $\nu_{\mu_1},\, \nu_{\mu_2}$ their respective stationary measures.
		Using the Furstenberg type formula.
		\begin{align*}
			|\ell(\mu_1)-\ell(\mu_2)| & \leq \left| \int_G \int_{\partial X} h_\xi(gx_0)d\nu_{\mu_1}(\xi)d\mu_1(g) - \int_G \int_{\partial X} h_\xi(gx_0)d\nu_{\mu_2}(\xi)d\mu_2(g)\right|  \\
			& \leq \left| \int_G \int_{\partial X} h_\xi(gx_0)d\nu_{\mu_1}(\xi)d\mu_1(g) - \int_G \int_{\partial X} h_\xi(gx_0)d\nu_{\mu_1}(\xi)d\mu_2(g)\right| \\
			& \hspace{1cm} + \left| \int_G \int_{\partial X} h_\xi(gx_0)d\nu_{\mu_1}(\xi)d\mu_2(g) - \int_G \int_{\partial X} h_\xi(gx_0)d\nu_{\mu_2}(\xi)d\mu_2(g)\right| 	\\
			& \leq \frac{\lambda^3}{\alpha \log(b)}W_\alpha(\mu_1, \mu_2) +  \frac{\lambda^3}{\alpha \log(b) (1-k)}W_\alpha(\mu_1, \mu_2).
		\end{align*}
	
	\end{proof}
	\section{Large Deviation estimates}
	
	In this section we obtain the large deviations. Although the method used is based in Nagaev's \cite{nagaev1957some}, we will apply Duarte and Klein's recipe \cite{duarte2016lyapunov}. In §\ref{theMethod} we describe the recipe and ready the ingredients laid by Duarte and Klein whilst §4.2 is devoted to proving the large deviations.

	Let us recall the reader once more that $X$ stands for a $\delta$-hyperbolic metric space with a basepoint $x_0$, $G$ for its groups of isometries and $b$ for a real number between $1$ and $2^{1/\delta}$.
	
	\subsection{The method}
	\label{theMethod}
	
	In this section let $\Sigma$ denote a metric space and $L^\infty(\Sigma)$ its space of Borel measurable functions bounded in the sup-norm $||\cdot ||_\infty$.
	
	\begin{definition}[Markov Kernel and Operator]
		A Markov kernel is a function $K:\Sigma \to \Prob(\Sigma)$ $\omega_0 \mapsto K(\omega_0, \cdot)$ such that for every Borel measurable set $E$ the function $\omega_0 \mapsto K(\omega, E)$ is measurable. Each Markov kernel $K$ determines a linear operator $Q_K:L^\infty(\Sigma)\to L^\infty(\Sigma)$ given by
		\begin{equation*}
			(Q_K \varphi)(\omega_0)=\int_\Sigma \varphi(\omega_1)K(\omega_0, d\omega_1).
		\end{equation*}
	
		One defines the iterated Markov kernels in a recursive manner: $K^1 = K$ and
		\begin{equation*}
			K^n(\omega_0,E) :=\int_{\Sigma} K^{n-1}(\omega_1, E)K(\omega_0, d\omega_1),
		\end{equation*}
		where $E\subset \Sigma$ stands for a borel measurable set and $n\geq 2$. 
	\end{definition}

	In the previous section we defined the Markov operator for a given measure $\mu \in \Prob_c(G)$. The name is motivated by this definition by picking in $\partial X$ the Markov kernel
	\begin{equation*}
		K(\xi, E) :=\int_G \delta_{g^{-1}\xi}(E) d\mu(g).
	\end{equation*}
	Just as in the random case we also have the relation $Q_K^n = Q_{K^n}$ in this more general setting.
	
	\begin{definition}[Stationary Measure]
		Given a Markov kernel $K$, a probability measure $\mu$ on $\Sigma$ is $K$-stationary if for every Borel measurable subset $E\in \Sigma$,
		\begin{equation*}
			\mu(E) = \int_\Sigma K(\omega_0, E)\mu(d\omega_0).
		\end{equation*}
	\end{definition}
	
	A Markov system is a pair $(K,\mu)$, where $K$ is a Markov kernel on some metric space $\Sigma$ and $\mu$ is a $K$-stationary probability measure.
	
	Given $(K,\mu)$ a Markov system on $\Sigma$, consider $\Omega = \Sigma^\mathbb{N}$ the space of sequences $\omega=(\omega_n)$ in $\Sigma$. The product space $\Omega$ is metrizable and its Borel $\sigma$-algebra  $\mathcal{B}$ is the product $\sigma$-algebra generated by the cylinders, that is, generated by the sets
	\begin{equation*}
		C(E_0,...,E_m) := \{ \omega \in \Omega \, : \, \omega_j \in E_j, \textrm{ for } 0\leq j \leq m\},
	\end{equation*}
	where $E_0,...,E_m $ are Borel measurable in $\Sigma$.
	
	The set of $\mathcal{F}$-cylinders forms a semi-algebra on which
	\begin{equation*}
		\mathbb{P}_\mu [C(E_0,...,E_m)]:=\int_{E_m}... \int_{E_0} \mu(d\omega_0) \prod_{j=1}^{m} K(\omega_{j-1},d\omega_j).
	\end{equation*}
	defines a pre-measure. By Carathéodory's extension theorem, it extends to a measure, still denoted $\mathbb{P}_\mu$ and often called the Kolmogorov extension, on $(\Omega, \mathcal{B})$.
	
	Given a random variable $\zeta:\Omega \to \mathbb{R}$, its expected value with respect to $\mu$ in $\Prob(\Sigma)$ is
	\begin{equation*}
		\mathbb{E}_\mu(\zeta) :=\int_\Omega \zeta d\mathbb{P}_\mu.
	\end{equation*}
	If $\mu$ is $\delta_{\omega_0}$ the Dirac measure at $\omega_0$, then we soften the notation by setting $\mathbb{P}_{\omega_0} = \mathbb{P}_{\delta_{\omega_0}}$.

	Consider a Markov system $(K, \mu)$ on a metric space $\Sigma$. Given some Borel measurable observable $\zeta:\Sigma \to \mathbb{R}$, let $\hat{\zeta}:\Omega \to \mathbb{R}$ be the Borel measurable function $\hat{\zeta}(\omega) = \zeta(\omega_0)$. We call a sum process of $\zeta:\Sigma \to \mathbb{R}$ the sequence of random variables $\{S_n(\zeta)\}$ on $(\Omega, \mathcal{B})$,
	\begin{equation*}
		S_n(\zeta)(\omega) := \sum_{i=0}^{n-1} \hat{\zeta} \circ T^i(\omega) = \sum_{i=0}^{n-1} \zeta(\omega_i),
	\end{equation*}
	where $T$ denotes the shift map.
	
	Let $(B, ||\cdot ||_B)$ be a Banach algebra such that $B\subset L^\infty(\Sigma)$, there exists a seminorm $\upsilon_B$ in $B$ such that $||\cdot||_B \leq ||\cdot ||_\infty + \upsilon_B$, for every $f\in B$, $\bar{f}$ and $|f|$ also belong to $B$ and so does $\mathbf{1}$, moreover $\upsilon_B(f)=0$ if and only if $f$ is constant. Denote by $L(B)$ the Banach algebra of bounded linear operators $Q:B\to B$ with both its operator norm $|||Q|||_B:= \sup \{ ||Qv||_B \, : \, ||v||_B = 1 \}$ 
	
	\begin{definition}
		We say that a Markov kernel acts simply and quasi-compactly on a Banach algebra $B$ as above if there are constants $C<\infty$ and $0<\sigma <1$	such that for every $f\in B$ and $n\geq 0$,
		\begin{equation*}
			\left| \left| Q_K^nf - \left(\int_\Sigma f d\mu \right)\mathbf{1} \right| \right|_B \leq C \sigma^n ||f||_B.
		\end{equation*}
	\end{definition}
	
	\begin{theorem}[Theorem 5.4 in \cite{duarte2016lyapunov}]
		\label{LDTT}
		Let $(K,\mu)$ be a Markov system acting simply and quasi-compactly on a Banach algebra $B\subset L^\infty(\Sigma)$ as above. Then given $\zeta \in B$, there are constants $k, \varepsilon_0>0$ and $C<\infty$ such that for all $\omega_0 \in \Sigma$, $0 < \varepsilon < \varepsilon_0$ and $n\in \mathbb{N}$
		\begin{equation*}
			\mathbb{P}_{\omega_0} \left[ \left| \frac{1}{n}S_n(\zeta) - \mathbb{E}_\mu(\zeta) \right| \geq \varepsilon \right] \leq Cb^{-k \varepsilon ^2n}.
		\end{equation*} 
		Moreover, the constants $C, k, \varepsilon_0$ depend on $|||Q_K|||_B$, $\upsilon_B(\zeta)$ as well as the constants $C$ and $\sigma$ controlling the simple and quasi-compact behaviour.
	\end{theorem}

	\subsection{Obtaining the Large Deviation Estimates}
	\label{deviations}
	
	Let $\mu\in \Prob_c(G)$, given $0 \leq \alpha \leq 1$ and $f\in L^\infty(G_\lambda \times \partial X)$, define
	\begin{align*}
		\upsilon_\alpha^\mu(f) & :=  \sup_{\substack{g\in \supp \mu,  \\ \xi \neq \eta \in \partial X} }  \frac{|f(g, \xi)-f(g,\eta)|}{D_b(\xi, \eta)^\alpha},\\
		||f||_\alpha & := ||f||_\infty + \upsilon_\alpha^\mu(f),
	\end{align*}
	and set 
	\begin{equation*}
		\mathcal{H}_\alpha(G_\lambda \times \partial X) : = \left\{\ f\in L^\infty(G_\lambda \times \partial X) \, : \, ||f||_\alpha < \infty \right\}.
	\end{equation*}
	the space of boundary Hölder continuous functions in $G_\lambda \times \partial X$ which once again is a Banach algebra with unity $\mathbf{1}$. In $G_\lambda \times \partial X$ we define the Markov kernel
	\begin{equation*}
		K_\mu(g, \xi, E) := \int_G \delta_{g',g^{-1}\xi}(E) d\mu (g'), 
	\end{equation*}
	which defines the Markov operator $Q_{K_\mu}:L^\infty(G_\lambda \times \partial X) \to L^\infty(G_\lambda \times \partial X)$
	\begin{equation*}
		(Q_Kf)(g, \xi) = \int_G f(g', g^{-1} \xi) d\mu(g').
	\end{equation*}
	
	We shall make use of the relation
	\begin{align*}
		(Q_{K_\mu}^nf)(g_0, \xi) & = \int_G (Q_{K_\mu}^{n-1}f)(g_1, g_0^{-1}\xi) d\mu(g_1) \\
		& = \int_G \int_G (Q_K^{n-2}f)(g_2, (g_0g_1)^{-1}\xi) d\mu(g_1)d\mu(g_2)\\
		&  = ... \\
		& = \int_G \int_G f(g_n, (g_0g)^{-1}\xi) d\mu^{n-1}(g) d\mu(g_n),
	\end{align*}
	to obtain the following Proposition.
	
	\begin{proposition}
		\label{Quasi-compact}
		Given $\lambda >0$ , for every $\mu \in \Prob_c(G_\lambda)$, $0<\alpha \leq 1$, $f\in \mathcal{H}_\alpha(G_\lambda \times \partial X)$ and $n \in \mathbb{N}$
		\begin{equation*}
			\upsilon_\alpha (Q_{K_\mu}^n f) \leq C(\delta)^\alpha \lambda^\alpha k_\alpha^{n-1}(\mu) \upsilon_\alpha(f).
		\end{equation*}
	\end{proposition}
	
	\begin{proof}
		Let  $g_0 \in \supp \mu$ and $\xi, \eta \in \partial X$,
		\begin{align*}
			\big|(Q_{K_\mu}^n f)(g_0, \xi) & -(Q_{K_\mu}^n f)(g_0, \eta) \big|  \leq \int_G \int_G \left| f(g_n, (g_0g)^{-1}\xi) - f(g_n, (g_0g)^{-1}\eta)  \right|d\mu^{n-1}(g) d\mu(g_n)  \\
			& \leq \upsilon_\alpha(f) \int_G D_b(g^{-1}g_0^{-1}\xi, g^{-1}g_0^{-1}\eta)^\alpha d\mu^{n-1}(g) \\
			& \leq \upsilon_\alpha(f) D_b(\xi, \eta)^\alpha  \left(\frac{D_b(g_0^{-1}\xi, g_0^{-1}\eta)}{D_b(\xi, \eta)} \right)^\alpha \int_G \left( \frac{D_b(g^{-1}g_0^{-1}\xi, g^{-1}g_0^{-1}\eta)}{D_b(g_0^{-1}\xi, g_0^{-1}\eta)} \right) ^\alpha d\mu^{n-1}(g) \\
			& \leq \upsilon_\alpha(f) k_\alpha^{n-1}(\mu)C(\delta)^\alpha b^{\alpha d(gx_0,x_0)}D_b(\xi, \eta)^\alpha.
		\end{align*}
		Passing $D_b(\xi, \eta)^\alpha$ to the left side and then taking the supremum over $\xi \neq \eta \in \partial X$ yields the result.
	\end{proof}
	
	Analogously to Proposition \ref{uniquenessProp}, given an irreducible measure $\mu\in \Prob_c(G_\lambda)$ with $\ell(\mu)>0$ there exists a unique $K_\mu$-stationary measure in $G_\lambda \times \partial X$. In particular it follows from an immediate computation that $\nu_\mu \in \Prob(\partial X)$ is the $\mu-$stationary measure if an only if $\mu \times \nu_\mu$ is the $K_\mu$-stationary.
	
	\begin{proof}[Proof of Theorem \ref{LDT} ]
		Let $\mu$ be an irreducible measure in $\Prob_c(G_\lambda)$ with $\ell(\mu)>0$ and $\nu_\mu \in \Prob(\partial X)$ be its stationary measure. Now take the set $\Omega \subset (G_\lambda \times \partial X)^\mathbb{N}$ consisting of sequences $\kappa = (\kappa_n) = (g_n, \xi_n)$ with $\xi_n = (g_0g_1...g_{n-1})^{-1} \xi_0$. Notice that $\Omega$ has full measure with respect to $\mathbb{P}_{\mu\times \nu_\mu}$. Next we consider the observables $\zeta \in L^\infty(G_\lambda \times \partial X)$
		\begin{equation*}
			\zeta(g, \xi) = h_\xi(gx_0),
		\end{equation*}
		whose sum, for each $\omega \in \Omega$ is
		\begin{align*}
			S_n(\zeta)(\omega) &= \sum_{i=0}^{n-1} \zeta(g_i, \xi_i)\\
			&= \sum_{i=0}^{n-1} h_{\xi_i}(g_ix_0)  \\
			& = \sum_{i=0}^{n-1} \omega^{-i} \cdot h_{\xi_0}(g_ix_0)\\
			&= \sum_{i=0}^{n-1} h_{\xi_0}(\omega^{i+1}x_0) - h_{\xi_0}(\omega^{i}x_0) \\
			& =  h_{\xi_0}(\omega^{n}x_0) - h_{\xi_0}(x_0) + \sum_{i=0}^{n-1} h_{\xi_0}(\omega^{i}x_0) - h_{\xi_0}(\omega^{i}x_0) \\
			& =  h_{\xi_0}(\omega^nx_0).
		\end{align*}

		By the Furstenberg type formula we also have
		\begin{equation*}
			\mathbb{E}_{\mu\times \nu_\mu} (\zeta) = \int_G \int_{\partial X} h_\xi(gx_0) d\nu_\mu(\xi) d\mu(g) = \ell(\mu), 
		\end{equation*}
		so we can prove the existence of large deviation estimates.
		
		By Proposition \ref{Quasi-compact}, $(K_\mu, \mu\times \nu_\mu)$ acts simply and quasi-compactly on $\mathcal{H}_\alpha(G_\lambda \times \partial X)$. Applying Theorem \ref{LDT} to the Markov kernel $(K, \mu\times \nu_\mu)$, there are constants  $\varepsilon_0, C, k>0$ such that for every $g_2\in V$, $0<\varepsilon<\varepsilon_0$, $(g_0, \xi) \in G_\lambda \times \partial X$ and $n\in \mathbb{N}$
		\begin{equation*}
			\mathbb{P}_{g_0} \left[ \left| \frac{1}{n} h_\xi(\omega^{n}x_0)- \ell(\mu) \right| > \varepsilon \right] \leq C b^{-k\varepsilon^2n},
		\end{equation*}
		which averaging over $g_0$ with respect to $\mu$ yields, for every $\xi \in \partial X$, 
		\begin{equation*}
			\mu^\mathbb{N} \left[ \left| \frac{1}{n} h_\xi(\omega^{n}x_0)- \ell(\mu) \right| > \varepsilon \right] \leq C b^{-k\varepsilon^2n}.
		\end{equation*}
		
		Now by Proposition \ref{Contract1}, $k_\alpha(\mu_1)^n<1$ for every $\mu_1$ neighbourhood of $\mu$, so again by Proposition \ref{Quasi-compact}, $(K_{\mu_1}, \mu_1 \times  \nu_{\mu_1})$ acts simply and quasi-compactly on $\mathcal{H}_\alpha(\supp \mu_1 \times \partial X)$. Now clearly $\upsilon_\alpha(\zeta) \leq 2\log \lambda$ and $|||Q_{K_{\mu_1}}|||_B\leq \max\{1, k_\alpha(\mu_1)\}$ which are all controlled by $\mu$, so we can keep the contants $C, k, \varepsilon_0$ in a neighbourhood of $\mu$.
	\end{proof}

	Corollary \ref{LDTdrift} is obtained through a direct application of Lemma \ref{comparisonLemma} and Theorem \ref{LDT}.

	\section{Furstenberg type formula}
	
	In this section we keep all the notations from before. Our goal is to prove Theorem \ref{furstenbergForm}, in fact the entire section is one big proof of this Theorem. We follow \cite{furstenberg1983random} closely whilst doing all the necessary adjustments. In this section compactness is once again essential to the argument so we go back to working with $X^h$ and then transport the result to $\partial X$. 
	
	Given $\mu\in \Prob_c(G)$ we begin by defining a Markov kernel in $\supp \mu \times X^h$
	\begin{equation*}
		K_\mu(g, h; E) := \int_G \delta_{(g',g^{-1}\cdot h)}(E) d\mu (g'), 
	\end{equation*}
	which defines the Markov operator $Q_{K_\mu}:L^\infty(\supp \mu \times  X^h) \to L^\infty(\supp \mu \times  X^h)$
	\begin{equation*}
		(Q_Kf)(g, h) = \int_G f(g', g^{-1} \cdot h) d\mu(g').
	\end{equation*}
	Just as before we take the set $\Omega \subset (\supp \mu \times X^h)^\mathbb{N}$ consisting of sequences $\kappa = (\kappa_n) = (g_n, h_n)$ with $h_n = (g_0g_1...g_{n-1})^{-1} \cdot h_0$. Notice that $\Omega$ has full measure with respect to $\mathbb{P}_{\mu\times \nu}$, where $\nu$ is any $\mu$-stationary measure. This time we take the observable $\zeta \in L^\infty(\supp \mu \times X^h)$ given by
	\begin{equation*}
		\zeta (g, h) = h(gx_0),
	\end{equation*}
	whose sum is $S_n(\zeta) = h_0(\omega^{n}x_0)$ for all $\omega \in \Omega$. As a matter of fact, $\zeta$ is actually continuous as we now prove.
	
	\begin{lemma}
		The observable $\zeta(g,h) = h(gx_0)$ with $(g,h)\in G\times X^h$ is continuous.
	\end{lemma}
	
	\begin{proof}
		Let $g_n \to g$ and $h_n \to h$ be converging sequences in $G$ and $X^h$, respectively. Then
		\begin{align*}
			\left| h_n(g_nx_0) - h(gx_0)  \right| & \leq \left| h_n(g_nx_0) - h_n(gx_0)  \right| + \left| h_n(gx_0) - h(gx_0)  \right| \\
			& \leq d(g_nx_0, gx_0) + \left| h_n(gx_0) - h(gx_0)  \right|.
		\end{align*}
		Due to the pointwise convergence $\left| h_n(g^{-1}x_0) - h(g^{-1}x_0)  \right| \to 0$ whilst, by definition of $D_b$, for $n$ large enough $D_b(g_nx_0, gx_0) = \log(b)d(g_nx_0, gx_0)$,  so continuity follows.
	\end{proof}
	
	Using Theorems 1.1 and 1.4 in \cite{furstenberg1983random} we now obtain the following result.
	
	\begin{theorem}
		\label{Furstenberg2.1}
		Let $\mu \in \Prob_c(G)$, for every horofunction $h\in X^h$ and $\mu^\mathbb{N}$-almost every $\omega$
		\begin{equation*}
			\limsup_{n \to \infty}\frac{1}{n}h(\omega^nx_0) \leq \sup \left\{ \int_G \int_{X^h} h(gx_0) d\nu(h)d\mu(g) \, : \, \mu \star \nu = \nu \right\} 
		\end{equation*}
		Moreover, if for all $\mu$-stationary measures $\nu$ the integral 
		\begin{equation*}
			\int_G \int_{X^h} h(gx_0) d\nu(h)d\mu(g)
		\end{equation*}
		takes the same value then for $\mu^\mathbb{N}$-almost every $\omega$
		\begin{equation*}
			\lim_{n \to \infty}\frac{1}{n}h(\omega^{n}x_0) = \int_G \int_{X^h} h(gx_0) d\nu(h)d\mu(g).
		\end{equation*}
	\end{theorem}
	
	To finish the proof of the Fürstenberg type formula we now prove that if $\mu$ is irreducible, then second part of the theorem holds. To that end, assume there is a $\mu-$stationary ergodic measure $\eta$ such that 
	\begin{equation*}
	    \int_G \int_{X^h} h(gx_0)d\nu(h)d\mu(g) := \beta_1 < \beta =: \sup \left\{ \int_G \int_{X^h} h(gx_0) d\nu(h)d\mu(g) \, : \, \mu \star \nu = \nu \right\} .
	\end{equation*}
	 Consider now the map  
    \begin{align*}
        F:G^\mathbb{N} \times X^h &\to G^\mathbb{N} \times X^h \\
        (\omega, h) &\mapsto (\sigma\omega, g(\omega_0, \omega_1)^{-1}\cdot h)
    \end{align*}
    which preserves the ergodic measure $\mu^\mathbb{N}\times\eta$. The observable $\zeta$ can be extended to $\bar \zeta : G^\mathbb{N} \times X^h\to\mathbb{R}$, $\overbar\zeta(\omega, h)=\zeta(\omega_0,\omega_1,h)$. Moreover, with this notation, $(S_n\zeta)(\omega) =\sum_{j=0}^{n-1} \overbar\zeta(F^j(\omega,h_0)) $ is a Birkhoff sum. By Birkhoff's ergodic theorem, for $\eta$-almost every $h_0\in X^h$ and $\mathbb{P}$-almost every $\omega\in\Omega$,
    \begin{equation*}
        \lim_{n\to\infty} \frac{1}{n} h_0(g^{(n)}(\omega) x_0) = \lim_{n\to\infty} \frac{1}{n} \sum_{j=0}^{n-1} \bar\zeta(F^j(\omega,h_0)) = \beta_1
    \end{equation*}
    which together with the HMET implies that $\beta_1=-\ell(g)$ and $h_0\in X_-^h(\omega)$. Next consider the family of sets
    \begin{equation*}
        S:= \left\{ h\in X^h \, :  \, h\in X_-^h(\omega) \,\,\mu^\mathbb{N}\textrm{-almost surely}\right\} .
    \end{equation*}
    The previous argument shows that $S\neq \emptyset$ for $\mu^\mathbb{N}$-almost every $\omega\in G^\mathbb{N}$.  Again by the HMET and the basic assumption
    the set $S_{\omega_0}$ must be a single horofunction
    $S={h}$, chosen measurably.
    The invariance of $X_-^h$ in the HMET now implies that $g\cdot S=S$, which proves that $g$ is not irreducible. This contradiction implies that the claim is true.
	
	As a final remark, every $\mu$ stationary measure in $\Prob_c(\partial X)$ can be pushed forward to a measure in $\Prob_c(X_\infty^h)$ through the local minimum, which we then extend to all $X^h$. By equivariance, this new measure is still $\mu$ stationary, but now in $X^h$. Therefore, result follows in $\partial X$ as stated in Theorem \ref{furstenbergForm}
	
	\section*{Acknowledgements}

    The author was supported by the University of Lisbon, under the PhD scholarship program: BD2018. I would also like to thank my advisor, professor Pedro Duarte for reading this text more times than he probably should have. Finally I am indebted to Cagri Sert for his very helpful comments as well as finding an error on the initial version of the text.

	\bibliographystyle{plain}
	\bibliography{biblio}
	
	Departamento de Matemática, Faculdade de Ciências, Universidade de Lisboa, Portugal
	
	Email: lmsampaio@fc.ul.pt

\end{document}